\documentclass[a4paper,11pt,reqno]{amsart}
\usepackage{amsmath,amsfonts,amsthm,amssymb,enumerate}
\usepackage{graphicx}
\usepackage{comment}

\numberwithin{equation}{section}

\newtheorem{theorem}{Theorem}[section]

\newtheorem{proposition}[theorem]{Proposition}

\newtheorem{lemma}[theorem]{Lemma}

\newtheorem*{fact*}{Fact}
{
\theoremstyle{definition}
\newtheorem{definition}[theorem]{Definition}
\newtheorem{example}[theorem]{Example}
\newtheorem{remark}[theorem]{Remark}
\newtheorem{claim}[theorem]{Claim}
}

\title{Surfaces in the 4-ball constructed using generators of knits and their graphical description}
\author{Inasa Nakamura}
\address{Department of Mathematics, Information Science and  Engineering, \newline
Saga University, 
1 Honjomachi, Saga, 840-8502, Japan \newline
TEL: +81-952-28-8113}
\email{inasa@cc.saga-u.ac.jp}

\author{Jumpei Yasuda}
\address{Department of Mathematics, Graduate School of Science, \newline Osaka University, 1-1, Machikaneyama, Toyonaka, Osaka 560-0043, Japan}
\email{u444951d@ecs.osaka-u.ac.jp}

\subjclass[2020]{Primary: 57K45, Secondary: 57Q35 57K10}
\keywords{surface in 4-space; surface braid; braided surface; tangle; knit; plat closure}

\begin{document}
\begin{abstract}
We introduce a new construction of surfaces in $D^2 \times B^2$, called knitted surfaces or BMW surfaces, which are described as the trace of deformations of knits. Here, knits are tangles obtained from classical braids from splicing at some crossings. Knitted surfaces are a generalization of braided surfaces. Further, we generalize charts of braided surfaces to BMW charts of knitted surfaces, which are finite graphs in $B^2$, and we show that a knitted surface has a BMW chart description.    
We show that every compact surface with non-empty boundaries properly embedded in $D^2 \times B^2$ is ambiently isotopic to some knitted surface: so such surfaces are described by BMW charts.

\end{abstract}
\maketitle

\section{Introduction}\label{sec1}
In this paper, surfaces are compact and tangles/surfaces are smoothly embedded. 
Let $I=[0,1]$, and let $n$ be a positive integer.
A (simple) {\it braided surface}
\cite{CKS,Kamada02} of degree $n$ is a surface $S$ properly embedded in $D^2 \times B^2$, where $B^2=I_3 \times I_4$  ($I_3=I_4=I$), such that the restriction $\pi|_S: S \to B^2$ is a (simple) branched covering map, where $\pi: D^2 \times B^2 \to B^2$ is the projection.
For $p: D^2 \times I_3 \times I_4 \to D^2 \times I_3$ and $t \in I_4$, the slice $S_{[t]}=p(S\cap (D^2 \times I_3 \times \{t\}))$ of a braided surface $S$ at $t$ is either a braid or a singular braid. The one-parameter family $\{S_{[t]}\}_{t \in I}$ gives deformations of a braid.
A braided surface is described by a finite graph called a {\it chart} in $B^2=I_3 \times I_4$, which is obtained as the \lq\lq trace'' of deformations of a braid word.

In this paper, we introduce a new construction of surfaces in $D^2 \times B^2$, called knitted surfaces or BMW surfaces, which are described as the trace of deformations of tangles called knits or BMW tangles. A knit of degree $n$ (or simply an $n$-knit) is a  tangle obtained from a braid of degree $n$ by splicing at some crossings. An $n$-knit is also called a BMW tangle, because $n$-knits generate the BMW  (Birman-Murakami-Wenzl) algebra \cite{Birman-Wenzl, Morton, Murakami}. 
The equivalence classes of $n$-knits form a monoid consisting of two types of generators $\sigma_i, \sigma_i^{-1}$ and $\tau_i$, where $\sigma_i$ and $\sigma_i^{-1}$ are a standard generator of the braid group $B_n$ and its inverse, and the additional generator $\tau_i$ is a pair of hooks between the $i$th and $(i+1)$th boundary points of the tangle as shown in Figure \ref{figure1}.

One of our main results is the following theorem.

\newtheorem*{MainTheorem}{Theorem~\ref{Theorem: Alexander theorem for surfaces in 4-ball}}
\begin{MainTheorem} 
    Every compact surface with non-empty boundaries properly embedded in $D^2 \times B^2$ is isotopic to some knitted surface. 
\end{MainTheorem}

\begin{figure}[ht]
\includegraphics*[height=3.5cm]{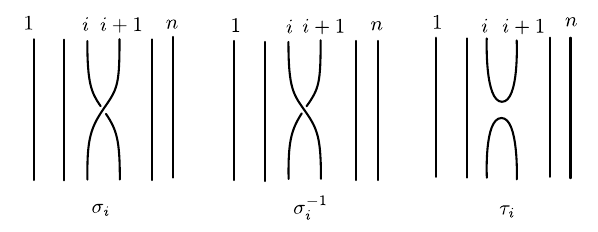}
\caption{Generators $\sigma_i$, $\sigma^{-1}$ and $\tau_i$ of the monoid $D_n$ of $n$-knits, where $i \in \{1,\ldots, n-1\}$. 
}\label{figure1}
\end{figure}

We introduce the notion of a graphical description called BMW charts or simply charts which present knitted surfaces (Definition \ref{def5-5}). A BMW chart is a finite graph in $B^2$, which is a generalization of Kamada's charts for simple braided surfaces \cite{Kamada02}--\cite{Kamada92}. 
We give a method of constructing a knitted surface $S(\Gamma)$ from a given BMW chart $\Gamma$.

\newtheorem*{MainTheorem2}{Theorem~\ref{thm4-6}}
\begin{MainTheorem2}
A knitted surface has a BMW chart description.
More precisely, for a knitted surface $S$, there exists a BMW chart $\Gamma$ such that $S$ and $S(\Gamma)$ are equivalent.
\end{MainTheorem2}

Thus, surfaces with non-empty boundaries can be investigated through BMW charts.
We define chart moves or C-moves consisting of T, CI, CII, CIII-moves, which are local modifications for charts (Definition \ref{def6-2}). 
We give a set of explicit CI-moves (Definition \ref{c1moves}), and we show the following.
\newtheorem*{MainTheorem3}{Theorem~\ref{thm7-3}}
\begin{MainTheorem3}
Two knitted surfaces of degree $n$ are equivalent 
if their presenting charts are related 
by a finite number of C-moves consisting of T, CII, CIII-moves and 
CI-moves of type (A1)--(F3). 
\end{MainTheorem3}
Further we consider knitted surfaces of degree 2, and we describe trivial surface-knots by charts of degree 2.

The paper is organized as follows.
In Section \ref{sec-knitted-surface}, we give the definition of knits and knitted surfaces.
In Section \ref{sec2}, we review the motion picture method by which we describe knitted surfaces.
 In Section \ref{sec5}, we give BMW charts. A knitted surface is presented by a BMW chart (Theorem \ref{thm4-6}). In Section \ref{sec6-0}, we show Theorem \ref{thm4-6}.
In Section \ref{sec-Alexander}, we prove Theorems \ref{Theorem: Alexander theorem for surfaces in 4-ball}. 
 In Section \ref{sec7-1}, we give local moves of BMW charts called BMW chart moves, 
 and we show Theorem \ref{thm7-3}.
In Section \ref{sec8}, we consider knitted surfaces of degree 2 and we give a BMW chart presentation of trivial surface-knots.
 We close the paper with further results which will be shown in our future papers in Section \ref{sec9}. 
 We give references \cite{CKS, Kamada02, Kawauchi} for the basics of 1-dimensional and 2-dimensional knot theory.

\section{Definition of knitted surfaces}\label{sec-knitted-surface}
In this section, we give a definition of knitted surfaces. In Section \ref{knit-str-tangle}, we give knits and knit structures, using the notion of pairings. In Section \ref{subsec-knitted-surface}, as an analogue to knits and knit structures, we give knitted surfaces and knit structures.
In this paper, we use notations as follow.
\begin{itemize}
  \item $I$: the interval $[0,1]$.
  \vspace{0.1cm}
    \item $D^2$: a 2-disk.
    \vspace{0.1cm}

    \item $B^2 = I_3 \times I_4$: the square, where $I_3 = I_4 = I$.
    \vspace{0.1cm}
    \item $n$: a positive integer.
    \vspace{0.1cm}
    \item $Q_n$: the fixed set consisting of $n$ interior points of $D^2$.
    \vspace{0.1cm}
    \item $y_0 = (0,0) \in \partial B^2$: the base point.
    \vspace{0.1cm}
    \item
    $p: D^2 \times I_3 \times I_4 \to D^2 \times I_3$: the projection.
    \vspace{0.1cm}
 \item
$S_{[t]}=p(S\cap D^2 \times I_3 \times \{t\})$: the slice of a surface $S$ in $D^2 \times B^2$ at $t$.
\vspace{0.1cm}
\item
$\displaystyle \cup \mathbf{X}=X_1 \cup X_2 \cup \cdots \cup X_m$: the union of elements for
$\mathbf{X}=\{X_1, \ldots, X_m\}$.
\end{itemize}

\subsection{Knits and knit structures}\label{knit-str-tangle}

Let $\beta$ be a tangle in $D^2\times I$, that is, a 1-manifold properly embedded in $D^2\times I$.
Let $\pi: D^2\times I \to I$ be the projection.
We consider $\pi|_{\beta}: \beta \to I$, the restriction map of $\pi$ by $\beta$.
A point $a \in \beta$ is called \textit{fold singular point} if $\pi|_{\beta}$ is the map sending $t$ to $t^2$ in appropriate coordinates around $a$ and $\pi(a)$.
Any tangle is deformed by slight perturbations so that $\pi|_{\beta}$ is non-degenerate and its critical points consist of fold singular points.

\begin{definition}
    Let $l$ be a simple arc embedded in $D^2\times I$ which intersects with $\beta$ transversally.
    Then, $l$ is called a \textit{pairing} of $\beta$ if it satisfies the following conditions.
    \begin{enumerate}
        \item We have $l \cap \beta = \partial l \cap \mathrm{Crit}(\beta) = \partial l$.
        \item The restriction map $\pi|_l: l \to I$ has no critical points.
        \item For each $a \in \partial l$, there exist local coordinates $(U; x)$ and $(V; t)$ around $a \in \beta$ and $\pi(a)\in I$ respectively,  such that $a=0\in U$, $\pi|_\beta(x) = x^2$ for any $x \in U$ and $\pi(l) \cap V = \{ t \in V \mid t\leq 0\}$.
    \end{enumerate}
\end{definition}

\begin{definition}
    Let $\beta$ be a tangle with non-degenerate $\pi|_\beta$.
    A \textit{knit structure} on $\beta$ is a set $\ell = \{ l_1, \dots, l_q\}$ of mutually disjoint pairings of $\beta$ such that all fold singular points are contained in $\cup \ell= l_1 \cup \dots \cup l_q$.
  When $\beta$ satisfies $\partial \beta = Q_n \times \{0, 1\}$ and  admits a knit structure $\ell$, we call the pair $(\beta, \ell)$ an \textit{$n$-knit}.
We remark that we include the case $\ell=\emptyset$.
\end{definition}

We remark that for a knit $(\beta, \ell)$, $\beta$ is a braid if and only if $\ell$ is the empty set.
For a knit structure $\ell = \{l_1, \ldots, l_q\}$ on $\beta$ and a self-homeomorphism $\varphi$ of $D^2 \times I$, we put $\varphi(\ell) = \{ \varphi(l_1), \ldots, \varphi(l_q)\}$ and $\varphi(\beta, \ell) = (\varphi(\beta), \varphi(\ell))$.

\begin{definition}
    Two $n$-knits $(\beta_0, \ell_0)$ and $(\beta_1, \ell_1)$ are said to be \textit{equivalent} if there exists an isotopy $\{\varphi_t\}_{t\in [0,1]}$ of $D^2\times I$ rel $\partial (D^2 \times I)$ carrying $\beta_0$ to $\beta_1$ such that $\varphi_1(\ell_0) = \ell_1$ and $\varphi_t(\beta_0, \ell_0)$ is an $n$-knit for any $t \in [0,1]$.
\end{definition}

We remark that when $(\beta_0, \ell_0)$ and $(\beta_1, \ell_1)$ are $n$-braids, then they are equivalent as knits if and only if they are equivalent as braids.
In this paper, we write an $n$-knit $(\beta, \ell)$ simply $\beta$ when the knit structure is clear.

Let $D_n$ be the set of equivalent classes of $n$-knits.
$D_n$ is a monoid by the multiplication induced from the composition of knits, which is called the \textit{knit monoid} of degree $n$.
Since an $n$-braid is an $n$-knit, $D_n$ contains the braid group $B_n$ as a submonoid.
When we consider in the framework of $n$-knits, we call $(Q_n \times I, \emptyset)$, the trivial $n$-braid, the {\it trivial $n$-knit}.

\begin{proposition}\label{Proposition: Presentation of Knit monoid}
    The knit monoid $D_n$ is generated by $3(n-1)$ elements, $\sigma_1, \ldots, \sigma_{n-1}$, $\sigma_1^{-1}, \ldots, \sigma_{n-1}^{-1}$, $\tau_1,\ldots, \tau_{n-1}$ with the following relations:
    \begin{align}
        & \sigma_i \sigma_{i}^{-1} = \sigma_i^{-1} \sigma_i = e, &  \label{eq2-1}\\
        & \sigma_i \tau_i = \tau_i \sigma_i = \tau_i,  & \label{eq2-2}\\
        & \sigma_i \tau_k = \tau_k \sigma_i, & (|i-k|>1), \label{eq2-3}\\
        & \tau_i \tau_k = \tau_k \tau_i, & (|i-k|>1), \label{eq2-4}\\
        & \sigma_i\sigma_k = \sigma_k \sigma_i, & (|i-k|>1), \label{eq2-5}\\
        & \sigma_i\sigma_j\sigma_i = \sigma_j\sigma_i\sigma_j, & (|i-j|=1), \label{eq2-6}\\
        & \sigma_i \sigma_j \tau_i = \tau_j \sigma_i \sigma_j, & (|i-j|=1). \label{eq2-7}
    \end{align}
    Here, $e \in D_n$ denotes the unit element represented by the trivial knit (trivial braid) $Q_n \times I$: $D_n$ has the monoid presentation
    \begin{align*}\label{knit-monoid}
        \left\langle
        \begin{array}{c|}
            \sigma_1, \ldots, \sigma_{n-1},\\ \sigma_1^{-1}, \ldots, \sigma_{n-1}^{-1},\\ \tau_1,\ldots, \tau_{n-1}
        \end{array}
        \begin{array}{l}
            \sigma_i \sigma_{i}^{-1} = \sigma_i^{-1} \sigma_i = e,~
            \sigma_i \tau_i = \tau_i \sigma_i = \tau_i,\\
            \sigma_i\sigma_k = \sigma_k \sigma_i,~
            \sigma_i \tau_k = \tau_k \sigma_i,~
            \tau_i \tau_k = \tau_k \tau_i \ \ (|i-k|>1), \\
            \sigma_i\sigma_j\sigma_i = \sigma_j\sigma_i\sigma_j,~
            \sigma_i \sigma_j \tau_i = \tau_j \sigma_i \sigma_j \ \ (|i-j|=1)
        \end{array}
        \right\rangle_{\mathrm{monoid}}.
    \end{align*}
\end{proposition}

\begin{proof}
We consider a singular braid, and we denote by $\dot{\sigma_i}$ the singular braid obtained from $\sigma_i$ by changing the crossing to the self-intersection point.
By changing $\dot{\sigma_i}$ to the hook pair $\tau_i$,
we see that the set of equivalence classes of singular $n$-braids is mapped onto that of $n$-knits. Hence, the monoid presentation of the singular $n$-braid monoid $SB_n$ \cite{Birman} induces the monoid generators and relations (\ref{eq2-1}) and (\ref{eq2-3})--(\ref{eq2-7}) of $D_n$.
By deforming each pairing of a knit to one point by isotopies, we obtain a singular braid.
This operation is unique up to the rotation of a hook around its pairing; so we have the monoid relation (\ref{eq2-2}), and relations (\ref{eq2-1})--(\ref{eq2-7}) generate the knit monoid $D_n$.
\end{proof}

\subsection{Knitted surfaces}\label{subsec-knitted-surface}

Let $S$ be a surface properly embedded in $D^2 \times B^2$.
Let $\pi: D^2 \times B^2 \to B^2$ be the projection, and let $\pi|_S: S \to B^2$ be the restriction map of $\pi$.
A point $a \in S$ is called a \textit{fold singular point} of $\pi|_S$ if there exist local coordinates $(U; x,y)$ and $(V; s,t)$ around $a\in S$ and $\pi(a) \in B^2$, respectively, such that $a = (0,0) \in U$, $V=\pi(U)$ and $\pi|_S(x,y)=(x, y^2)$ for any $(x,y) \in U$.
We denote by $\mathrm{Fold}(S)$ the set of fold singular points of $\pi|_S$.
A point $a \in S$ is called a \textit{branch point} of $\pi|_S$ if $\pi|_S$ sends $(x, y)\in U$ to $(x^2-y^2, 2xy) \in V$ for appropriate coordinates $(U; x,y)$ and $(V; s,t)$ around $a=(0,0)\in U $ and $\pi(a) \in V$.

\begin{definition}
Let $S$ be a surface properly embedded in $D^2 \times B^2$.
    Let $L$ be an orientable surface embedded in $D^2 \times B^2$ which intersects with $S$ transversally.
    Then $L$ is called a \textit{pairing} of $S$ if it satisfies the following.
    \begin{enumerate}
        \item The intersection $L \cap S $ satisfies $L \cap S=\partial L \cap \mathrm{Fold}(S)$, which is a non-empty set, and $\partial L \setminus \mathrm{Fold}(S)$ consists of simple open arcs in $D^2\times \partial B^2$.
        \item The restriction map $\pi|_L: L \to B^2$ has no critical points.
        \item For each $a \in \mathrm{Fold}(S)$, there exist local coordinates $(U; x,y)$ and $(V; s,t)$ around $a\in S$ and $\pi(a) \in B^2$, respectively, such that $a = (0,0) \in U$, $V=\pi(U)$, \[
      V=   \begin{cases}  \{(s,t) \mid -1 < s,t < 1\} & (\text{if \ } \pi(a) \notin \partial B^2)\\
          \{(s,t) \mid 0 \leq s<1, -1 < t < 1\} &(\text{if \ }\pi(a) \in \partial B^2)\\
            \end{cases}
        \]
      and $\pi|_{S}(x,y) = (x,y^2)$ for any $(x,y) \in U$ and $\pi(L \cap U) ~=~
                \{(s,t)\in V  ~|~ t \leq 0\}$.
    \end{enumerate}
\end{definition}

\begin{definition}
Let $S$ be a surface properly embedded in $D^2 \times B^2$ with non-degenerate $\pi|_S$ such that the singular points consist of fold singular points and branch points. A \textit{knit structure} on $S$ is a set $\mathcal{L} = \{L_1, \ldots, L_q\}$ of mutually disjoint pairings
     of $S$ such that $(\cup\mathcal{L}) \cap (D^2 \times \{y_0\}) = \emptyset$, and each fold singular point is contained in one of $\cup\mathcal{L}$, where
    $\cup\mathcal{L} = L_1 \cup \cdots \cup L_q$. When $S$ satisfies $S \cap (D^2\times \{y_0\}) = Q_n \times \{y_0\}$ and admits a knit structure $\mathcal{L}$, we call the pair $(S, \mathcal{L})$
    a \textit{knitted surface} of degree $n$. We remark that we include the case $\mathcal{L}=\emptyset$.
\end{definition}

\begin{definition}
    Two knitted surfaces $(S_0, \mathcal{L}_0)$ and $(S_1, \mathcal{L}_1)$ are said to be \textit{equivalent} if there is an isotopy $\{\Phi_t\}_{t\in[0,1]}$ of $D^2\times B^2$ rel $D^2\times \{y_0\}$, carrying $S_0$ to $S_1$ such that
         $(S_t, h_t) = (\Phi_t(S_0), \Phi_t(\mathcal{L}_0))$ is a knitted surface for any $t \in [0,1]$.
\end{definition}

We remark that a knitted surface $(S, \emptyset)$ of degree $n$ is a simple braided surface of degree $n$ \cite{CKS}.

 \begin{remark}
 The notion of a knit was introduced by Murakami \cite{Murakami}. The original notion of a knit in $D^2 \times I$ is not equipped with a knit structure: two $n$-knits $\beta$ and $\beta'$ are defined to be equivalent if they are isotopic rel $\partial$. With this equivalence, a knit is also called a {\it BMW tangle}, and the generators of the knit monoid generate the BMW (Birman-Murakami-Wenzl) algebra \cite{Birman-Wenzl, Morton}. Hence a knit $(\beta, l)$ with a knit structure will be also called a {\it BMW tangle}, and a knitted surface will be also called a {\it BMW surface}.
 \end{remark}

\section{Description of surfaces in the 4-ball}\label{sec2}
In this section, we review basic tools to treat surfaces in 4-space \cite{CKS, Kamada02}.
 
\subsection{Band surgery of a tangle along a band set}
Let $\beta$ be a tangle in $D^2 \times I$.
Let $B$ be a 2-disk $B$ in $D^2 \times I$ such that $B \cap \beta$
consists of a pair of intervals in $\partial B$. We call $B$ a {\it band} attaching to $\beta$. The result of {\it band surgery} along $B$ is the tangle
$\mathrm{Cl}((\beta\cup \partial B)\backslash \beta \cap B)$, denoted by $h(\beta; B)$. Here, since tangles are smoothly embedded, by a slight perturbation we assume that $h(\beta; B)$ is also smoothly embedded.

Let $\mathcal{B}=\{B_1, \ldots, B_m\}$ be a set of mutually disjoint bands attaching to $\beta$. Then we denote by $h(\beta; \mathcal{B})$ the result of band surgery along bands of $\mathcal{B}$:
\[
h(\beta, \mathcal{B})=\mathrm{Cl}((\beta\cup \partial (\cup \mathcal{B}))\backslash \beta \cap (\cup \mathcal{B})),
\]
where $\cup \mathcal{B}=B_1 \cup B_2 \cup \cdots \cup B_m$.

\subsection{Horizontal product of surfaces}
Let $S$ be a surface properly embedded in $D^2  \times B^2$, and we denote $B^2=I_3 \times I_4$, where $I_3=I_4=[0,1]$. Let $p: D^2 \times I_3 \times I_4 \to D^2 \times I_3$ be the projection.
We call the image by $p$ of the intersection $S \cap (D^2 \times I_3 \times \{t\})$ ($t \in I_4$) the {\it slice} of $S$ {\it at $t$}, denoted by $S_{[t]}$:
\[
S_{[t]}=p(S\cap D^2 \times I_3 \times \{t\}).
\]
We call $t$ the {\it $I_4$-coordinate} of the slice $S_{[t]}$.

Let $S_1$ and $S_2$ be  surfaces properly embedded in $D^2 \times B^2$ satisfying $(S_1)_{[1]}=(S_2)_{[2]}$.
The \textit{horizontal product} of $S_1$ and $S_2$, denoted by $S_1 \cdot S_2$, is the surface $S$ given by
\begin{align*}
    S_{[t]} ~=~ \begin{cases}
        (S_1)_{[2t]}  & (0\leq t \leq \frac{1}{2}),\\
        (S_2)_{[2t-1]}  & (\frac{1}{2} \leq t \leq 1).
    \end{cases}
    \end{align*}

   For surfaces $S_1, S_2, S_3$ which admit the horizontal products $(S_1 \cdot S_2) \cdot S_3$ and $S_1 \cdot (S_2 \cdot S_3)$, the horizontal products are isotopic. In particular, when the surfaces are knitted surfaces of degree $n$, $(S_1 \cdot S_2) \cdot S_3$ and $S_1 \cdot (S_2 \cdot S_3)$ are equivalent knitted surfaces. We denote by $S_1 \cdot S_2 \cdots S_m$ a horizontal product of surfaces $S_1, S_2, \ldots, S_m$.
\subsection{Motion picture method}
Let $S$ be a surface properly embedded in $D^2  \times B^2$. 
We call a point in $S$ a {\it saddle point} (respectively a {\it minimal/maximal point}) if it is a saddle point (respectively minimal/maximal point) of the function $\pi|_S: S \to I_4$, where $\pi: D^2 \times I_3 \times I_4 \to I_4$ is the projection. 
In this paper,
we modify saddle points and minimal/maximal points of $S$  into saddle bands and minimal/maximal disks, respectively;
see \cite[Section 8.4]{Kamada02}.

When we have an  isotopy $\{\beta_t\}_{t\in I}$ of tangles in $D^2 \times I_3$, the union $\cup_{t \in I} (\beta_t \times \{t\})$ forms a properly embedded surface $S$ in $D^2 \times I_3 \times I$ such that $S_{[t]}=\beta_t$ ($t \in I$). We call the constructed surface $S=\cup_{t \in I} (\beta_t \times \{t\})$ the {\it trace} of the isotopy of tangles $\{ \beta_t\}_{t\in I}$.
We also use the term \lq\lq trace'' for diagrams.
The diagram of $S$ is described as the trace of diagrams of slices of $S$.

When we describe $S$ in $D^2 \times B^2$ by figures, we use the {\it motion picture method}.
Let $\beta_i$ ($i=1,\ldots, m$) be a subset of $D^2 \times I_3$.
A {\it motion picture} is a finite sequence $\beta_0 \to \beta_1 \to \cdots \to \beta_m$ consisting of the following sequences.
\begin{enumerate}
\item
A sequence $\beta_{j-1} \to \beta_j$, where $\beta_{j-1}$ and $\beta_j$  are isotopic tangles in $D^2 \times I_3$. We often denote the sequence by $\beta_{j-1} \overset{\cong}{\rightarrow} \beta_j$.
\item
A sequence $\beta_{j-1} \to \beta_j \to \beta_{j+1}$, where
$\beta_{j-1}$ and $\beta_{j+1}$ are tangles, and
\begin{align*}
\beta_j& =\beta_{j\pm1} \cup ( \cup \mathcal{B}) \cup \mathbf{D}, 
\end{align*}
where $\mathcal{B}$ is a band set attaching to $\beta_{j\pm1}$ and $\mathbf{D}$ is a union of mutually disjoint 2-disks such that the intersection  $ \mathbf{D} \cap \beta_{j \pm1}$ consists of embedded circles in $\beta_{j \pm 1}$.
\end{enumerate}

We define the surface presented by a motion picture, as follows.
For the motion picture $\beta_{j-1} \to \beta_j$ of (1), let $\{\beta'_t\}_{t \in I}$ be an isotopy connecting $\beta_{j-1}$ and $\beta_j$: $\beta'_0=\beta_{j-1}$ and $\beta'_1=\beta_j$. Then, we define a surface $S$ {\it presented by $\beta_{j-1} \to \beta_j$} as the trace of $\{\beta'_t\}_{t \in I}$ :
\[
S_{[t]}=\beta'_t, \quad (t \in I).
\]
We remark that the presented surface is unique up to isotopy. When $\beta_{j-1}$ and $\beta_j$ are $n$-knits, we require that $\{\beta'_t\}_{t \in I}$ is an isotopy of $n$-knits; thus the presented surface is unique up to equivalence of knitted surfaces.

For the motion picture $\beta_{j-1} \to \beta_j \to \beta_{j+1}$ of (2), we define the surface {\it presented by $\beta_{j-1} \to \beta_j \to \beta_{j+1}$} as the surface $S$ given by:
\[
S_{[t]}=\begin{cases}
 \beta_{j-1}         & (0\leq t <\frac{1}{2}),\\
        \beta_j =\beta_{j\pm1} \cup ( \cup \mathcal{B}) \cup \mathbf{D},        & (t=\frac{1}{2}),\\
        \beta_{j+1}  & (\frac{1}{2}<t \leq 1).\\
\end{cases}
\]
For a motion picture $\beta_0 \to \beta_1 \to \cdots \to \beta_m$, we define the presented surface as a horizontal product of the surfaces presented by the sequences of (1) and (2).
 
 \begin{remark}
 We consider a motion picture $\beta_{j-1} \to \beta_j \to \beta_{j+1}$ of (2) such that $\beta_j$ is as follows.
 We denote by $S$ the surface presented by $\beta_{j-1} \to \beta_j \to \beta_{j+1}$.

\noindent
(Around a saddle band) The case $\beta_j =\beta_{j\pm1} \cup B$, where $B$ is a band attaching to $\beta_{j-1}$. Then, $B$ is a saddle band of $S$.
The surface $S$ is given by:
\[
S_{[t]}=\begin{cases}
 \beta_{j-1}         & (0 \leq t < \frac{1}{2}),\\
        \beta_{j-1} \cup B        & (t=\frac{1}{2}),\\
        h(\beta_{j-1}; B)  & (\frac{1}{2}<t \leq 1). \\
\end{cases}
\]

\noindent
(Around a maximal disk)
The slice $\beta_j$ is the union of $\beta_{j-1}$ and a disk bounded by a simple closed circle in $\beta_{j-1}$. In other words, the motion picture is given by the sequence  $\beta_{j-1} \to \beta_{j-1} \cup D \to \beta_{j-1}\backslash \partial D$, where $D$ is a maximal disk such that $\partial D$ is a component of $\beta_{j-1}$:
\[
S_{[t]}=\begin{cases}
 \beta_{j-1}         & (0 \leq t <\frac{1}{2}),\\
        \beta_{j-1} \cup D        & (t=\frac{1}{2}),\\
        \beta_{j-1} \backslash \partial D  & (\frac{1}{2}<t \leq 1).  \\
\end{cases}
\]

\noindent
(Around a minimal disk)
The slice $\beta_j$ is the union of $\beta_{j+1}$ and a disk bounded by a simple closed circle in $\beta_{j+1}$. In other words, the motion picture is given by the sequence  $\beta_{j-2} \to \beta_{j-2} \cup d \to \beta_{j-2}\cup \partial d$, where $d$ is a minimal disk such that $\partial d$ is a component of $\beta_{j+1}$:
\[
S_{[t]}=\begin{cases}
 \beta_{j-1}         & (0 \leq t < \frac{1}{2}),\\
        \beta_{j-1} \cup d        & (t=\frac{1}{2}),\\
        \beta_{j-1} \cup \partial d  & (\frac{1}{2}<t \leq 1). \\
\end{cases}
\]

\end{remark}

In this paper, for slices, we treat knits with/without bands and disks, and we describe geometric knits by words in the knit monoid.

\subsection{Cell moves}
    Let $S$ be an $n$-manifold embedded in $\mathbb{R}^{n+2}$.
    Let $C$ be a $(n+1)$-ball in $\mathbb{R}^{n+2}$ such that $S \cap \partial C$ is an $n$-disk in $S$. We call $S$ an {\it $n$-cell}. Let $S'$ be an $n$-manifold in $\mathbb{R}^{n+2}$ given by
    \[
    S'=\mathrm{Cl}((S \cup \partial C) \backslash (S \cap \partial C)).
    \]
    Here, since we treat smoothly embedded surfaces, by a slight perturbation we assume that $S'$ is also smoothly embedded.
     The $n$-manifolds $S$ and $S'$ are isotopic.
    We say that $S'$ is obtained from $S$ by an {\it $n$-cell move along $C$}. 
    In this paper, we treat the case $n=1$ and $2$: 2-cell moves and 3-cell moves. 
    
\section{Normal form and BMW chart description} \label{sec5}

In this section, we give a normal form of a knitted surface and of degree $n$ (Proposition \ref{prop5-2}), and a BMW chart of degree $n$ (Definition \ref{def5-5}).
A BMW chart contains two types of vertices: $\sigma$-vertices, which are vertices of a chart of a 2-dimensional braid, and newly added vertices.
We use an argument similar to the case of 2-dimensional braids \cite{Kamada02, Kamada96, Kamada94}.
We modify the saddle points and minimal/maximal points of knitted surfaces into saddle bands and minimal/maximal disks, repsectively.
Let $n$ be a positive integer.
We denote by a word in the knit monoid a geometric knit, that is a representative of the isotopy class associated with the word,
and we denote by $e$ the trivial knit $Q_n \times I$.

\begin{definition}\label{def-elem-transf}
We call the following sequences of $n$-knits  {\it elementary transformations} of $n$-knits, where $\beta \leftrightarrow \beta'$ denotes the sequences $\beta \to \beta'$ and $\beta' \to \beta$:
\begin{align}
(1) \ & e \leftrightarrow \sigma_i^\epsilon, &\label{eq5-1}  \\
\  & e \leftrightarrow \tau_i, & \label{eq5-2} \\
 (2) \ &\tau_i \tau_i \leftrightarrow \tau_i, & \label{eq5-12}\\
(3) \ & \tau_i \leftrightarrow \sigma_i \tau_i, & \label{eq5-13}\\
\ &
 \tau_i \leftrightarrow \tau_i \sigma_i, & \label{eq5-14}\\
(4)\  & \sigma_i \sigma_i^{-1} \leftrightarrow e, \quad  \sigma_i^{-1} \sigma_i \leftrightarrow e, & \label{eq5-4}\\
& \sigma_i \sigma_j \leftrightarrow \sigma_j \sigma_i & (|i-j|>1), \label{eq5-9}\\
& \sigma_i \tau_j \leftrightarrow \tau_j \sigma_i & (|i-j|>1), \label{eq5-10}\\
& \tau_i \tau_j \leftrightarrow \tau_j \tau_i & (|i-j|>1), \label{eq5-11}\\
& \sigma_i \sigma_j \sigma_i \leftrightarrow \sigma_j \sigma_i \sigma_j  & (|i-j|=1),\label{eq5-5}\\
& \sigma_i \sigma_j \tau_i \leftrightarrow \tau_j \sigma_i \sigma_j  & (|i-j|=1), \label{eq5-6}
\end{align}
where $\epsilon \in \{+1, -1\}$, $i,j \in \{1,\ldots, n-1\}$.

We define the surface $S$ {\it presented by an elementary transformation}, as follows. We denote by $\beta \leftrightarrow \beta'$ the elementary transformation.
\begin{enumerate}

\item $e \leftrightarrow \sigma_i^\epsilon$, $e \leftrightarrow \tau_i$.

The surface $S$ is the surface presented by the motion picture $\beta \to \beta \cup B \to h(\beta; B)$, where
$B$ is a band; see Figure \ref{fig5}. For the case $e \leftrightarrow \sigma_i^\epsilon$, $B$ corresponds to the crossing $\sigma_i^\epsilon$, and for the case $e \leftrightarrow \tau_i$, $B$ is a band which connects the $i$th and $(i+1)$th strings of $e$.

The motion pictures presenting
(\ref{eq5-1})--(\ref{eq5-2}) for $\epsilon=+1$ is as in Figure \ref{fig5}. For the case $\epsilon=-1$, the motion pictures are obtained as the mirror image of Figure \ref{fig5}, with respect to $D^2 \times \{1/2\}$. 

\begin{figure}[ht]
\includegraphics*[height=3cm]{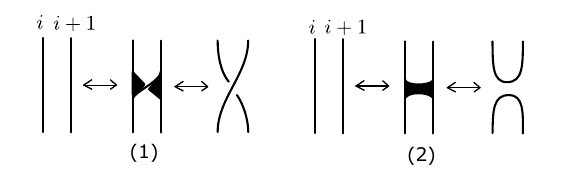}
\caption{The motion pictures of the knitted surfaces presented by (1) $e \leftrightarrow \sigma_i$ and (2)
$e \leftrightarrow \tau_i$.
}\label{fig5}
\end{figure}

\item
$\tau_i \tau_i \leftrightarrow \tau_i$.

The surface $S$ is the surface presented by the motion picture $\tau_i \tau_i \to \tau_i \cup D \to \tau_i $, where
$D$ is a disk such that $\partial D$ is the embedded circle in $\tau_i \tau_i$; see Figure \ref{fig11}.

\begin{figure}[ht]
\includegraphics*[height=3cm]{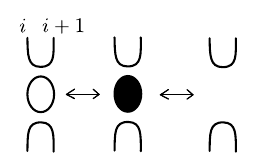}
\caption{The motion picture of the knitted surface presented by $\tau_i \tau_i \leftrightarrow \tau_i$, where we omit the $j$th strings for $j \neq i, i+1$.}
\label{fig11}
\end{figure}

\item $\tau_i \leftrightarrow \sigma_i \tau_i$,
 $\tau_i \leftrightarrow \tau_i \sigma_i$.

The surface $S$ is a surface presented by the motion picture $\beta \overset{\cong}{\to} \beta'$ such that
the diagrams of slices $\{S_{[t]}\}$ are associated with a Reidemeister move of type I.

\item
These transformations are relations of the knit monoid $D_n$ (Proposition \ref{Proposition: Presentation of Knit monoid}).
The surface $S$ is a surface presented by the motion picture $\beta \overset{\cong}{\to} \beta'$ which satisfies the following.

\noindent
(4a) $\sigma_i \sigma_i^{-1} \leftrightarrow e$, \quad  $\sigma_i^{-1} \sigma_i \leftrightarrow e$.

The diagrams of slices $\{S_{[t]}\}$ are associated with a Reidemeister move of type II.
 The motion picture presenting
(\ref{eq5-4}) for $\epsilon=+1$ is as in Figure \ref{fig6}.

 \noindent
 (4b) $\sigma_i \sigma_j \leftrightarrow \sigma_j \sigma_i$, $\sigma_i \tau_j \leftrightarrow \tau_j \sigma_i$,
 $\tau_i \tau_j \leftrightarrow \tau_j \tau_i$ \quad $(|i-j|>1)$.

The motion pictures presenting
(\ref{eq5-9})--(\ref{eq5-11}) are as in Figure \ref{fig10}. The diagrams of $\{S_{[t]}\}_{t \in I}$
are related by an ambient isotopy of the plane. At $t=1/2$, the slice $S_{[1/2]}$ is not presented by a word in the knit monoid $D_n$, while other slices are presented by words in $D_n$.

\begin{figure}[ht]
\includegraphics*[height=2.5cm]{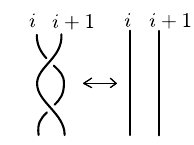}
\caption{A Reidemeister move of type II  presented by $\sigma_i \sigma_i^{-1} \leftrightarrow e$,
where we omit the $j$th strings for $j \neq i, i+1$. }
\label{fig6}
\end{figure}

\begin{figure}[ht]
\includegraphics*[height=7cm]{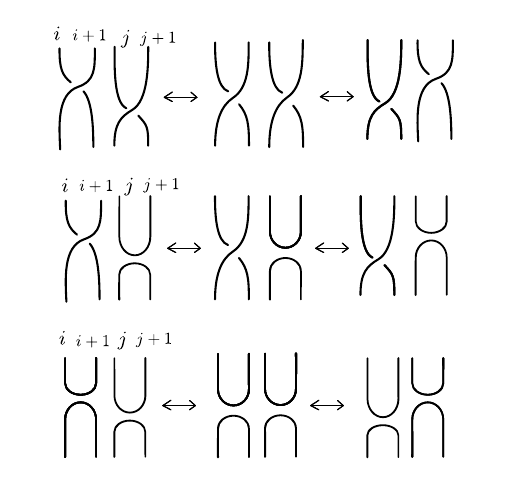}
\caption{Motion pictures of knitted surfaces presented by $\sigma_i \sigma_j \leftrightarrow \sigma_j \sigma_i$ (top figure), $\sigma_i \tau_j \leftrightarrow \tau_j \sigma_i$ (lower left figure) and $\tau_i \tau_j \leftrightarrow \tau_j \tau_i$ (lower right figure), $|i-j|>1$, where $i<j$ and we omit the $k$th strings for $k \neq i, i+1,  j, j+1$.}
\label{fig10}
\end{figure}

\noindent
(4c)
$\sigma_i \sigma_j \sigma_i \leftrightarrow \sigma_j \sigma_i \sigma_j$, $\sigma_i \sigma_j \tau_i \leftrightarrow \tau_j \sigma_i \sigma_j$ \quad $(|i-j|=1)$.

The diagrams of slices $\{S_{[t]}\}$ are associated with a Reidemeister move of type III (respectively two Reidemeister moves of type II) for the first case (respectively the second case).
For (\ref{eq5-5}), the diagram of the knitted surface (or the braided surface) contains a triple point.
 The motion picture presenting the first case
(\ref{eq5-5}) for $\epsilon=+1$ and $j=i+1$ is as in Figure \ref{fig7}.

\begin{figure}[ht]
\includegraphics*[height=3cm]{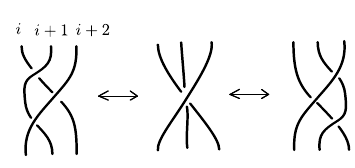}
\caption{A Reidemeister move of type III presented by  $\sigma_i \sigma_j \sigma_i \leftrightarrow \sigma_j \sigma_i \sigma_j$ ($|i-j|=1$), for $j=i+1$,
where we omit the $k$th strings for $k \neq i, i+1, i+2$.}
\label{fig7}
\end{figure}

\end{enumerate}
Given an elementary transformation $\beta \to \beta'$ of $n$-knits, the surface presented by $\beta \to \beta'$ is
a knitted surface of degree $n$, which will be called an {\it elementary knitted surface}.
\end{definition}

\begin{definition}\label{Definition: knitted surface in normal form}
    A knitted surface $S$ of degree $n$ is said to be in a \textit{normal form} if it is a horizontal product of elementary knitted surfaces.

  Namely,  knitted surface $S$ of degree $n$ is in a normal form if
    there exists a partition $0=t_0<t_1< \ldots <t_{m}=1$ for some $m$ such that for slices $\beta_j = S_{[t_j]}=p( S \cap (D^2 \times I_3 \times \{t_j\}))$, we have the following.
    \begin{enumerate}
        \item Each slice $\beta_j$ is an $n$-knit $(j=0, \ldots, m)$.

        \item Each sequence $\beta_{j-1} \to \beta_{j}$ is  an elementary transformation, and $S\cap[t_{j-1}, t_j]$ is a surface presented by $\beta_{j-1} \to \beta_{j}$  $(j=1, \ldots, m)$.
        \end{enumerate}
\end{definition}

\begin{proposition}\label{prop5-2}
    Every knitted surface is equivalent to a knitted surface in a normal form.
\end{proposition}

\begin{proof}
Let $S$ be a knitted surface.
By equivalence, we deform $S$ so that the pairings are sufficiently thin, and the resulting surface, also denoted by $S$, satisfies the condition as follows: there exists a set $J$ of finite number of inner points in $I_4$ such that
for any $s \in J$, each slice $S_{[s]}$ is described by one of the followings.
\begin{enumerate}
    \item The union of a knit and a sufficiently small attaching band (respectively disk), that is a saddle band (respectively a minimal/maximal disk).
    \item The union of a knit and a sufficiently small attaching band that contains the preimage of a branch point.
\end{enumerate}
For any interval $[t,t']$ of $\mathrm{Cl}(I_4 \backslash N(J))$, the diagrams of $S_{[t]}$ and $S_{[t']}$ are related by a finite sequence of Reidemeister moves of types I, II and III and ambient isotopies of the plane.
Divide $[0,1]$ into subintervals $[t_{j-1}, t_j]$ with $0=t_0<t_1<\cdots<t_m=1$ such that we have the followings.
\begin{enumerate}
    \item[(3)] Each $s \in J$ is the middle point of some $[t_{j-1}, t_j]$.
    \item[(4)] For each interval $[t_{j-1}, t_j]$ that does not contain elements of $J$, we have one of the followings.
    \begin{enumerate}
        \item[(4a)] The diagrams of $\beta_{j-1}$ and $\beta_j$ are related by a Reidemeister move.
        \item[(4b)] The diagrams of $\beta_{j-1}$ and $\beta_j$ are related by an ambient isotopy of the plane  that changes words of the knit.
    \end{enumerate}
\end{enumerate}
Here, $\beta_j=S_{[t_j]}=p(S\cap (D^2 \times I_3 \times \{t_j\}))$.
From now on, we consider the case of
$\beta_{j-1} \neq \beta_j$.
For Case (4b), since the ambient isotopy keeps the slices in the form of knits, this is described by a sequence consisting of exchange of two letters: $w_i w_j \leftrightarrow w_j w_i$ $(w_i \in \{\sigma_i, \sigma_i^{-1}, \tau_i \mid i=1, \ldots, n-1\})$ with $|i-j|>1$.
We divide $[t_{j-1}, t_j]$ into subintervals and denote the resulting division by the same notation so that we have (3), and one of either (4a) or (4b') as follows.
\begin{enumerate}
    \item[(4b')] The diagrams of $\beta_{j-1}$ and $\beta_j$    are related by one of the following: (\ref{eq5-9})--(\ref{eq5-11}) in Definition \ref{def-elem-transf} and (\ref{eq5-18})--(\ref{eq5-19}) in Lemma \ref{lem5-3}.
\end{enumerate}

Let $B$ be a band attaching to $\tau_i$ such that $h(\tau_i; B)= \sigma_i^\epsilon$, where $\epsilon \in \{+1, -1\}$. We define the knitted surface presented by
\begin{equation}\label{tuika}
\tau_i \leftrightarrow \sigma_i^\epsilon \quad (\epsilon \in \{+1, -1\}).
\end{equation}
as the surface given by the motion picture $\tau_i \to \tau_i \cup B \to e$ or $e \to \tau_i \cup B \to \tau_i$.

For $[t_{j-1}, t_j]$ that contains an element $s \in J$, if the slice at $s$ contains a saddle band or a disk containing the preimage of a branch point, since such bands and disks are sufficiently small, $\beta_{j}$ is obtained from $\beta_{j-1}$ by band surgery associated with (\ref{eq5-1}), (\ref{eq5-2}), or  (\ref{tuika}).
The knitted surface presented by (\ref{tuika}) is equivalent to the knitted surface presented by
\begin{equation}
\tau_i \to \sigma_i^\epsilon \tau_i \to \sigma_i^\epsilon, \  \text{or} \  \ \sigma_i^\epsilon \to \sigma_i^\epsilon \tau_i \to \tau_i, \label{eq1014}
\end{equation}
which is a combination of (\ref{eq5-2}) and (\ref{eq5-21}).
The equivalence is shown by considering an isotopy of knitted surfaces that carries the band of (\ref{tuika}) to the band of $\tau_i \leftrightarrow  \sigma_i^\epsilon \tau_i$ of (\ref{eq1014}).
If the slice at $t$ contains a minimal/maximal disk, since such a disk is sufficiently small,
 $\beta_{j}$ is obtained from $\beta_{j-1}$
by pasting a disk at $(t_{j-1}+t_j)/2$ as described by (\ref{eq5-12}).
For $[t_{j-1}, t_j]$ that satisfies (4a), we have the following cases.
\begin{enumerate}
    \item[(Case 1)] A Reidemeister move is of type I.
    Since we must keep slices in the form of knits, this is presented by (\ref{eq5-20})--(\ref{eq5-21}) in Lemma \ref{lem5-3}.
    By Lemma \ref{lem5-3}, it is presented by a sequence consisting of (\ref{eq5-4}) and (\ref{eq5-13})--(\ref{eq5-14}).
    \item[(Case 2)] A Reidemeister move is of type II.
    In this case, we have (\ref{eq5-4}).
    \item[(Case 3)] A Reidemeister move is of type III.
    In this case, the move is one of (\ref{eq5-15})--(\ref{eq6-17}) in Lemma \ref{lem5-3}. By Lemma \ref{lem5-3}, it is presented by a sequence consisting of (\ref{eq5-4})--(\ref{eq5-6}).
\end{enumerate}

For $[t_{j-1}, t_j]$ that satisfies (4b'), by Lemma \ref{lem5-3}, this is presented by a sequence consisting of (\ref{eq5-4}) and (\ref{eq5-9})--(\ref{eq5-11}).

Thus, by taking a division of $[t_{j-1}, t_j]$ into subintervals and denoting the resulting division of $I_4=[0,1]$ by the same notation $0=t_0< t_1\cdots < t_m=1$, we have $S$ in a normal form.
\end{proof}

\begin{lemma}\label{lem5-3}
The Reidemeister moves presented by $(\ref{eq5-20})$, $(\ref{eq5-21})$, $(\ref{eq5-18})$, $(\ref{eq5-19})$, $(\ref{eq5-15})$--$(\ref{eq5-16})$, and $(\ref{eq6-17})$, as follow, are realized by the moves $(\ref{eq5-4})$ and the corresponding $(\ref{eq5-13})$, $(\ref{eq5-14})$, $(\ref{eq5-9})$, $(\ref{eq5-10})$, $(\ref{eq5-5})$ and $(\ref{eq5-6})$ in Proposition \ref{prop5-2}, respectively:
\begin{align}
& \tau_i \leftrightarrow \sigma_i^\epsilon \tau_i, \label{eq5-20} & \\
& \tau_i \leftrightarrow \tau_i \sigma_i^\epsilon, \label{eq5-21}& \\
& \sigma_i^\epsilon \sigma_j^\delta \leftrightarrow \sigma_j^\delta \sigma_i^\epsilon   &  (|i-j|>1), \label{eq5-18}\\
& \sigma_i^\epsilon \tau_j \leftrightarrow \tau_j \sigma_i^\epsilon  &  (|i-j|>1), \label{eq5-19}\\
&\sigma_i^\epsilon \sigma_j^\epsilon \sigma_i^{\epsilon} \leftrightarrow \sigma_j^{\epsilon} \sigma_i^{\epsilon} \sigma_j^{\epsilon}  &   (|i-j|=1), \label{eq5-15}\\
&\sigma_i^\epsilon \sigma_j^{\epsilon} \sigma_i^{-\epsilon} \leftrightarrow \sigma_j^{-\epsilon} \sigma_i^{\epsilon} \sigma_j^{\epsilon}   &   (|i-j|=1), \label{eq5-16}\\
&\sigma_i^\epsilon \sigma_j^{\epsilon} \tau_i \leftrightarrow \tau_j \sigma_i^{\epsilon} \sigma_j^\epsilon   &  (|i-j|=1), \label{eq6-17}
\end{align}
where $\epsilon, \delta \in \{+1, -1\}$, $i,j \in \{1,\ldots, n-1\}$.
\end{lemma}

\begin{proof}

We have
\begin{eqnarray*}
\sigma_i^{-1} \sigma_j^{-1} \tau_i &\leftrightarrow& \sigma_i^{-1} \sigma_j^{-1} \tau_i (\sigma_j \sigma_i \sigma_i^{-1} \sigma_j^{-1})\\
&=& \sigma_i^{-1} \sigma_j^{-1} (\tau_i \sigma_j \sigma_i) \sigma_i^{-1} \sigma_j^{-1}\\
&\leftrightarrow& \sigma_i^{-1} \sigma_j^{-1} (\sigma_j \sigma_i\tau_j)\sigma_i^{-1} \sigma_j^{-1} \\
&=& (\sigma_i^{-1} \sigma_j^{-1} \sigma_j \sigma_i) \tau_j \sigma_i^{-1} \sigma_j^{-1}\\
&\leftrightarrow& \tau_j \sigma_i^{-1} \sigma_j^{-1};
\end{eqnarray*}
hence we have (\ref{eq6-17}).
The other cases can be shown similarly; see \cite[Lemma 5.3]{N} for a detailed proof.
\end{proof}

\begin{definition}\label{def5-5}
Let $\Gamma$ be a finite graph in $B^2$.
Then, $\Gamma$ is a {\it BMW chart} of degree $n$, or simply a {\it chart} of degree $n$, if it satisfies the following conditions.

\begin{enumerate}
\item
The intersection $\Gamma \cap \partial B^2$ consists of a finite number of endpoints of edges of $\Gamma$ meeting $\partial B^2$ orthogonally. Though they are vertices of degree one, we call elements of $\Gamma \cap \partial B^2$ {\it boundary points}, and we call only a vertex in $\mathrm{Int}(B^2)$ {\it a vertex of $\Gamma$}.

\item
Each edge is equipped with a label in $\{1,\ldots, n-1\}$ and moreover, each edge is either oriented or unoriented. We call an edge with an orientation (respectively without an orientation) a {\it $\sigma$-edge} (respectively {\it $\tau$-edge}).

\item
Each vertex is of degree 1, 3, 4 or 6 as follows; see Figure \ref{BMWchart}.
\end{enumerate}

\begin{enumerate}[(i)]
\item[]
\hspace{-1cm}

\item
A vertex of degree one such that the vertex is  connected with an edge which has a  label $i$ for some $i \in \{1,\ldots, n-1\}$ and equipped with/without an orientation. We call it a {\it black vertex}. We depict a black vertex by a small disk, and if it is connected with a $\sigma$-edge (respectively $\tau$-edge), that is, if it has (respectively does not have) an orientation, then we call it a {\it black $\sigma$-vertex} (respectively a {\it black $\tau$-vertex}).

\item
A vertex of degree 3 such that around the vertex, the three edges have the label $i$ for some $i\in \{1,\ldots, n-1\}$ and either all of three edges are unoriented or one of three edges is oriented. 
We call it a {\it trivalent $\tau$-vertex} or a \textit{mixed trivalent vertex}, respectively, and we call both types a {\it trivalent vertex}. 

\item
A vertex of degree 4 such that around the vertex, each pair of diagonal edges
has the same label in $\{1, \ldots, {n-1
}\}$, and each pair has a coherent orientation or does not have an orientation, and the labels of the two pairs, $i$ and $j$,  satisfy $|i-j|>1$. We call the vertex a {\it crossing}; in particular, we call the vertex consisting of $\sigma$-edges (respectively $\tau$-edges) a {\it $\sigma$-crossing} (respectively a {\it $\tau$-crossing}).

\item
A vertex of degree 6 such that around the vertex, the six edges have labels $i$and $j$ alternately clockwise ($i,j \in \{1,\ldots, n-1\}$), such that $|i-j|=1$. And three consecutive edges (respectively the first and the second edges) have an orientation toward the vertex and the other consecutive edges (respectively the 4th and the 5th edges) have an orientation from the vertex.  We depict the vertex by a small circle. 
We call the vertex a {\it white vertex}; in particular, we call the vertex consisting of $\sigma$-edges a {\it white $\sigma$-vertex}.
\end{enumerate}

\end{definition}

\begin{figure}[ht]
\includegraphics*[height=8.5cm]{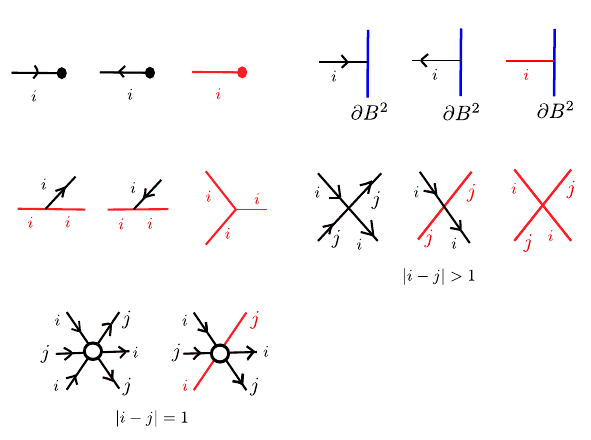}
\caption{Vertices of a BMW chart.}
\label{BMWchart}
\end{figure}

\begin{definition}
For a chart $\Gamma$ in $B^2$, a {\it height function} is a surjective continuous map $\phi: B^2 \to [0,1]$ satisfying the following conditions.

\begin{enumerate}
\item
For any $t \in [0,1]$, the preimage $\phi^{-1}(t)$ is homeomorphic to $I$.
\item
For any $t \in [0,1]$, $\Gamma \cap \phi^{-1}(t)$ consists of a finite number of points.
\end{enumerate}

Further, we call a height function $\phi$ a {\it good height function} if it satisfies the following condition.

\begin{enumerate}
\item[(3)]
For any vertex $v$ of $\Gamma$ and for a sufficiently small $\epsilon>0$,
\[
|\# (\Gamma \cap N_{t+\epsilon/2})-\#(\Gamma \cap N_{t-\epsilon/2})| \leq 1, \]
where $t=\phi(v)$ and $N_s=\phi^{-1}(s) \cap N$ for an $\epsilon$-neighborhood $N$ of $v$.
\end{enumerate}
\end{definition}

\begin{claim}\label{claim0308}
For a chart $\Gamma$ in $B^2$, there exists a height function.
\end{claim}

For a chart $\Gamma$, we construct a surface $S(\Gamma)$ using a height function $\phi$ as in Section \ref{sec6-2}. We remark that a knitted surface in a normal form is $S(\Gamma)$ that is constructed from a chart $\Gamma$ and a good height function $\phi: I_3 \times I_4 \to I_4$, where $\phi$ is given by the projection to the second factor.
\\

From a knitted surface of degree $n$, we obtain a BMW chart of degree $n$ (Section \ref{sec6-1}), and
from a given BMW chart $\Gamma$,
we give a method to construct a surface $S(\Gamma)$  (Section \ref{sec6-2}).

\begin{theorem}\label{thm4-6}
A knitted surface has a BMW chart description.
More precisely, for a knitted surface $S$, there exists a BMW chart $\Gamma$ such that $S(\Gamma)$ is a knitted surface and $S$ and $S(\Gamma)$ are equivalent.
\end{theorem}

\begin{remark}\label{rem0308}
Knitted surfaces include the surfaces isotopic to the following: 
\begin{enumerate}
\item
The surface presented by
\begin{equation}
\tau_i \leftrightarrow \sigma_i^\epsilon, \label{eq320}
\end{equation}
where one is obtained from the other by attaching a band and
where $\epsilon \in \{+1, -1\}$, $i \in \{1,\ldots, n-1\}$. .

\item
The surface presented by the following:
\begin{eqnarray}
&& \tau_i \tau_j \tau_i \leftrightarrow \tau_i \label{eq321}
\\
&& \tau_i \sigma_j^\epsilon \tau_i \leftrightarrow \tau_i \label{eq322}
\\
&& \sigma_i^\epsilon \sigma_j^\epsilon \tau_i \leftrightarrow \tau_j \tau_i \label{eq323}\\
&& \tau_i \tau_j \tau_i \leftrightarrow \tau_i \sigma_j^\epsilon \tau_i, \label{eq324}
\end{eqnarray}
where one is obtained from the other by an isotopy of $D^2 \times I_3$ rel $\partial$ and $|i-j|=1$, $\epsilon \in \{+1, -1\}$, $i,j \in \{1,\ldots, n-1\}$.

\end{enumerate}

The surface (\ref{eq324}) is a combination of (\ref{eq321}) and (\ref{eq322}).
It suffices to see that the surfaces (\ref{eq320})--(\ref{eq323}), (1)--(4) as in Figure \ref{20240324-2}, are isotopic to the knitted surfaces presented by charts (1)--(4) in Figure \ref{20240308-1}, respectively.

The knitted surface presented by the chart (1) in Figure \ref{20240308-1} is as in Figure \ref{20240324-3}. By an isotopy, we can deform the band to a twisted band as in (1) in Figure \ref{20240324-2}, hence the knitted surface is isotopic to the surface (\ref{eq320}).
The knitted surface presented by the chart (2) in Figure \ref{20240308-1} is as in the upper surface in Figure \ref{20240325-1}. By a 3-cell move, we see that the surface is isotopic to the lower surface in Figure \ref{20240325-1}; hence it is isotopic to the surface (\ref{eq321}).
The case for the surface (3) is shown by a similar way to the surface (2).
The knitted surface presented by the chart (4) in Figure \ref{20240308-1} is as in Figure \ref{20240325-2}, that is isotopic to the surface (\ref{eq323}).
\end{remark}

\begin{figure}[ht]
\includegraphics*[height=55mm]{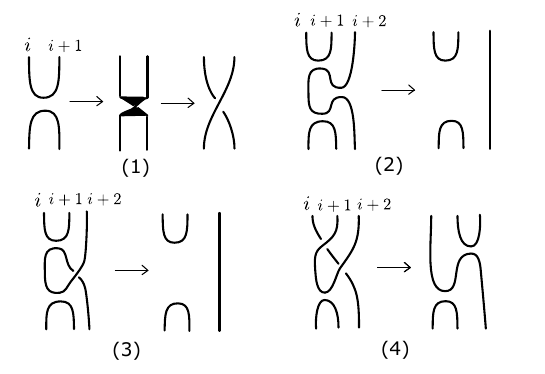}
\caption{The surfaces in Remark \ref{rem0308}, where $i<j$ and $\epsilon=+1$, and we omit the $k$th strings for $k \neq i, i+1$ or $k \neq i,  j, j+1$.}
\label{20240324-2}
\end{figure}

\begin{figure}[ht]
\includegraphics*[height=6cm]{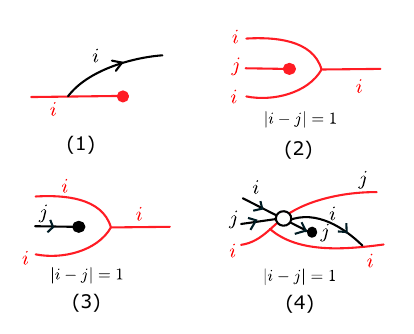}
\caption{BMW charts presenting the surfaces in Remark \ref{rem0308}, where $\epsilon=+1$.}
\label{20240308-1}
\end{figure}

\begin{figure}[ht]
\includegraphics*[height=2.5cm]{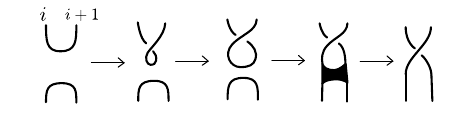}
\caption{The knitted surface presented by the chart (1) in Figure \ref{20240308-1}, where we omit the $k$th strings for $k \neq i, i+1$.}
\label{20240324-3}
\end{figure}

\begin{figure}[ht]
\includegraphics*[height=5.5cm]{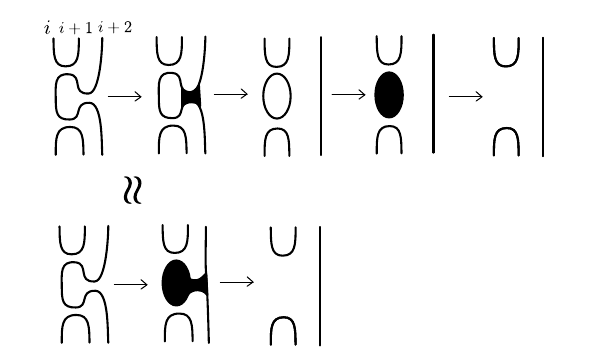}
\caption{The knitted surface presented by the chart (2) in Figure \ref{20240308-1}, where $i<j$ and we omit the $k$th strings for $k \neq i,  j, j+1$.}
\label{20240325-1}
\end{figure}

\begin{figure}[ht]
\includegraphics*[height=5.5cm]{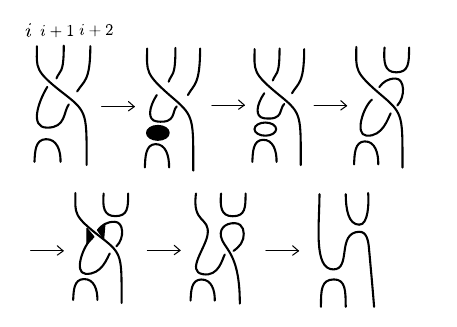}
\caption{The knitted surface presented by the chart (4) in Figure \ref{20240308-1}, where $i<j$ and we omit the $k$th strings for $k \neq i,  j, j+1$.}
\label{20240325-2}
\end{figure}

\section{Proof of Theorem \ref{thm4-6}}\label{sec6-0}
\subsection{From a knitted surface to a BMW chart}\label{sec6-1}
Let $S$ be a knitted surface of degree $n$ in $D^2 \times I_3 \times I_4$. Let $\phi: D^2 \to [0,1]$ be a good height function.
We identify $\phi$ with the projection $I_3 \times I_4 \to I_4$.

We obtain a BMW chart $\Gamma$ in $I_3 \times I_4$ from $S$ as follows.
By equivalence, we may assume that $S$ is in a normal form.
Namely, there exists a partition $0=t_0<t_1< \ldots <t_{m}=1$ satisfying the conditions of Definition~\ref{Definition: knitted surface in normal form}.
For each $j \in \{0,\ldots, m\}$, let $w_1 w_2 \cdots w_{n_j}$ be the word of $\beta_j$ $(w_k \in \{\sigma_i, \sigma_i^{-1}, \tau_i \mid i=1,\ldots,n-1\})$.
Take $n_j$ points in $I_3 \times \{t_j\}$ and assign $i$ to the $k$th point with respect to the orientation of $I_3$, where $i$ is the index of $w_k$, that is, $w_k$ is one of $\sigma_i, \sigma_i^{-1}, \tau_i$.
Then, we connect the points in $I_3 \times \{t_{j-1}\}$ and $I_3 \times \{t_{j}\}$ by arcs and vertices in $I_3 \times [t_{j-1}, t_{j}]$ in such a way as follows, where $j=1, \ldots, m-1$.
We connect the points that do not concern (\ref{eq5-1})--(\ref{eq5-11}) by mutually disjoint parallel arcs, and those concerning (\ref{eq5-1})--(\ref{eq5-12}), (\ref{eq5-5})--(\ref{eq5-11}) by arcs connected with
a vertex in $I_3 \times \{(t_{j-1}+t_{j})/2\}$.
Now, the points concerning (\ref{eq5-4}) are a pair of points in $I_3 \times \{t_{j-1}\}$ or $I_3 \times \{t_j\}$; and we connect the pair of points by a simple arc.
By regarding the points in $I_3 \times \{t_{j}\}$ $(j=1, \ldots, m-1)$ as inner points of edges, we obtain a graph.
Assign each edge with the label of the belonging points, and moreover, if the point in $I_3 \times \{t_j\}$ comes from a letter $\sigma_i$ (respectively $\sigma_i^{-1}$) in $\beta_j$, then assign the edge with an orientation coherent with that of $I_4$ (respectively opposite to that of $I_4$).  Then, by depicting vertices as those of a BMW chart, we have a BMW chart of degree $n$.

\begin{remark}
Let $D^2=I_1 \times I_2$ $(I_1=I_2=I)$ and
let $\pi: D^2 \times B^2 \to I_2 \times B^2$ be a regular projection and let $\pi(S)$ be the diagram of $S$.
Then, a double point curve of $\pi(S)$ is projected by the projection $\pi': I_2 \times B^2 \to B^2$ to an edge of $\Gamma$ with an orientation. A triple point of $\pi(S)$ is projected by $\pi'$ to a white $\sigma$-vertex, and a branch point of $\pi(S)$ is projected to a black $\sigma$-vertex or a mixed trivalent vertex.
\end{remark}

\subsection{From a BMW chart $\Gamma$ to the surface $S(\Gamma)$ presented by $\Gamma$}\label{sec6-2}

Given a BMW chart $\Gamma$ of degree $n$ in $B^2$, we construct the surface $S$, denoted by $S(\Gamma)$.
We take a height function $\phi: B^2 \to [0,1]$ of $\Gamma$.

For $t \in [0,1]$, $\Gamma \cap \phi^{-1}(t)$ consists of a finite number of points such that either
\begin{enumerate}
\item
the intersection $\Gamma \cap \phi^{-1}(t)$ consists of transverse intersection points of edges of $\Gamma$ and $\phi^{-1}(t)$, or
\item
the intersection $\Gamma \cap \phi^{-1}(t)$ consists of vertices, critical points of $\phi|_\Gamma$, and transverse intersection points of edges of $\Gamma$ and $\phi^{-1}(t)$.
\end{enumerate}
We remark that we regard boundary points of $\Gamma$ in $\partial B^2$ as transverse intersection points of edges of $\Gamma$ and $\partial B^2$.

For Case (1), we take a knit $\beta_t$ presented by
$w_1 w_2 \cdots w_{m}$, where $m$ is the number of points of $\Gamma \cap \phi^{-1}(t)$, and $w_k=\sigma_i$ (respectively $\sigma_i^{-1}$) if the $k$th point with respect to the direction of $\phi^{-1}(t)$ is a point of an oriented edge with the label $i$ and an orientation that is coherent to the normal vector to $\phi^{-1}(t)$ (respectively an orientation opposite to the normal vector to $\phi^{-1}(t)$), and $w_k=\tau_i$ if the $k$th point is a point of an unoriented edge with the label $i$ ($k=1, \ldots, m$).
We take the knit $\beta_t$ as the slice $S \cap (D^2 \times \phi^{-1}(t))$. Further, we arrange so that a crossing of $\beta_t$ is projected to the corresponding point of $\Gamma$ by the projection to $B^2$, and the pairing $l$ of a hook pair is sufficiently small and the middle point of $l$ is projected to the corresponding point of $\Gamma$ by the projection to $B^2$.

Case (2) occurs for a finite number of $t$.
So using slices near $S \cap (D^2 \times \phi^{-1}(t))$, we construct an embedded surface.

Thus we have the surface $S$ in $D^2 \times B^2$ presented by a chart $\Gamma$.
By construction, the surface $S$ is a knitted surface whose chart is $\Gamma$.

\subsection{From the surface $S(\Gamma)$ to a knitted surface in a normal form}\label{sec6-3}

For a chart $\Gamma$, we take a height function $\phi$. We identify $\phi$ with the projection $I_3 \times I_4 \to I_4$. By modifying $\phi$ if necessary, we assume that
$S(\Gamma)$ satisfies the condition of a normal form except for the surface presented by $\tau_i \tau_i \leftrightarrow e$.

Let $\Gamma'$ be a chart obtained from $\Gamma$ by exchanging each arc as in the left figure of Figure \ref{Fig17} (or its mirror image where we put the mirror in a vertical position) to a subgraph (or its mirror image) as in the right figure of Figure \ref{Fig17}, where, in Figure \ref{Fig17}, the horizontal direction is the $I_4$ direction. The graph as in the right figure of Figure \ref{Fig17} consists of a pair of arcs $c_1, c_2$ and one edge $c_3$ such that $c_3$ is connected with a vertex of degree one and $c_1, c_2, c_3$ are connected with a vertex of degree three.
Then, the surface $S'$ constructed from $\Gamma'$ is a knitted surface in a normal form: thus $S'$ is a knitted surface whose chart $\Gamma'$ presents $S'$.
By Proposition \ref{prop6-1}, the surface $S$ in question is equivalent to $S'$. Thus, the surface $S(\Gamma)$ constructed from any chart $\Gamma$ is equivalent to a knitted surface in a normal form, and hence $S(\Gamma)$ is a knitted surface.

\begin{proposition}\label{prop6-1}
The surfaces presented by the BMW charts of degree $n$ as in the left figure and the right figure of Figure \ref{Fig17} are equivalent.
\end{proposition}

From a cylinder $D^2 \times I$, we obtain a solid torus $D^2 \times S^1$ by identifying $D^2 \times \{0\}$ and $D^2 \times \{1\}$.We denote the map $D^2 \times I \to D^2 \times S^1$ by $f$. For a knit $\beta$ in $D^2 \times I$, the {\it closure} of $\beta$ in $D^2 \times S^1$ is the image $f(\beta)$.

\begin{proof}[Proof of Proposition \ref{prop6-1}]
We denote by $N$ the disk in which BMW charts are drawn.
The surface $S$ presented by the chart as in the left figure in Figure \ref{Fig17} is the one as depicted in the middle figure of Figure \ref{Fig23}, that is obtained by deforming the surface whose slices are $\tau_i$ as in the left figure in Figure \ref{Fig23} by an ambient isotopy of $D^2 \times N$ induced from that of $N$. The surface $S$ consists of a pair of disks whose boundaries are the closure of the knit $\tau_i\tau_i$ in the solid torus $D^2 \times \partial N$, where we ignore other components coming from the $j$th strings for $j \neq i, i+1$.
This surface $S$ is related with the surface $S'$ as in the right figure of Figure \ref{Fig23} by an isotopy of surfaces. Hence, $S$ is equivalent to $S'$, that is the knitted surface presented by the chart as in the right figure of Figure \ref{Fig17}.
Hence we have the required result.
\end{proof}

\begin{figure}[ht]
\includegraphics*[height=3cm]{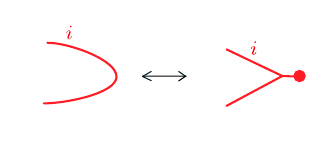}
\caption{Charts of degree $n$ whose associated surfaces are equivalent, where $i \in \{1, \ldots, n-1\}$ and we omit the 2-disks in which the charts are drawn.}
\label{Fig17}
\end{figure}

\begin{figure}[ht]
\includegraphics*[height=5.5cm]{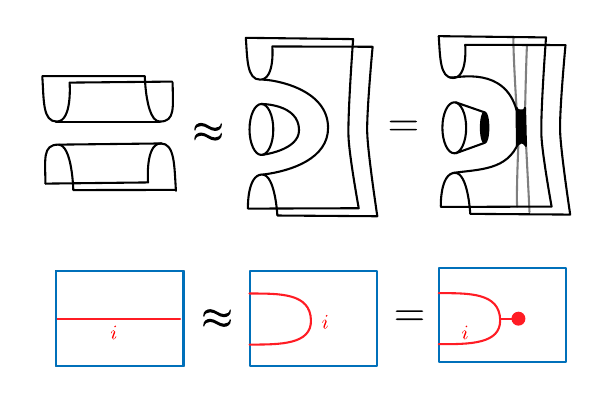}
\caption{The surfaces presented by charts in Figure \ref{Fig17}  are equivalent, where we omit other components coming from the $j$th strings for $j \neq i, i+1$.}
\label{Fig23}
\end{figure}

\begin{proof}[Proof of Theorem \ref{thm4-6}]
By the argument in Sections \ref{sec6-1}--\ref{sec6-3}, we have the required result.
\end{proof}

\section{Proof of Theorem~\ref{Theorem: Alexander theorem for surfaces in 4-ball}}\label{sec-Alexander}

In this section, we show the following theorem:
\begin{theorem}\label{Theorem: Alexander theorem for surfaces in 4-ball}
    Every compact surface with non-empty boundaries properly embedded in $D^2 \times B^2$ is isotopic to some knitted surface.
\end{theorem}

The proof of Theorem~\ref{Theorem: Alexander theorem for surfaces in 4-ball} is similar to the proof of \cite[Theorem~1.1]{Yasuda21}.
For a set $\mathbf{X}=\{X_1, \ldots, X_m\}$, we denote by $\cup \mathbf{X}$ the union of elements $X_1 \cup X_2 \cup \cdots \cup X_m$.

\subsection{The plat closure of braids, Hilden subgroup, and adequate braids}
In this subsection, we recall the notions of the plat closure of braids. An $n$-braid is an $n$-knit whose knit structure is the empty set.  

A \textit{wicket} is a semi-circle $\omega$ properly embedded in $D^2 \times I$ such that $\omega$ meets $D^2 \times \{0\}$ orthogonally.
A \textit{configuration of $n$ wickets} is a set of mutually disjoint $n$ wickets.
We denote by $\mathcal{W}_n$ the set of configurations of $n$ wickets.
Let $\omega_0$ be a configuration of $n$ wickets such that $\partial (\cup \omega_0) = Q_{2n} \times \{0\}$ and $\{q_{2i-1}\} \times \{0\}$ and $\{q_{2i}\} \times \{0\}$ are the boundary points of the same component of $\omega_0$ for each $i \in \{1, \dots, n\}$, where $\cup \omega_0$ denotes the union of wickets of $\omega_0$. Let $\omega_0^* \subset D^2 \times I$ be the mirror image of $\omega_0$ with respect to $D^2 \times \{1/2\}$. 
We remark that $\cup \omega_0$ and $\cup \omega_0^*$ are disjoint and $(\cup \omega_0) \cup (\cup \omega_0^*)$ represents a $2n$-knit $\tau = \tau_1 \tau_3 \dots \tau_{2n-1}$. 
For a $2n$-braid $\beta$, the \textit{plat closure} of $\beta$, denoted by $\mathrm{pcl}(\beta)$, is the link in $\mathbb{R}^3$ in the form of the product $\cup \omega_0^* \cdot \beta \cdot \cup \omega_0 \subset D^2 \times I \subset \mathbb{R}^3$, where  
the product is defined similarly to the product of braids; see Figure~\ref{20240707-1}.

\begin{figure}[h]
    \centering
    \includegraphics[width=0.7\hsize]{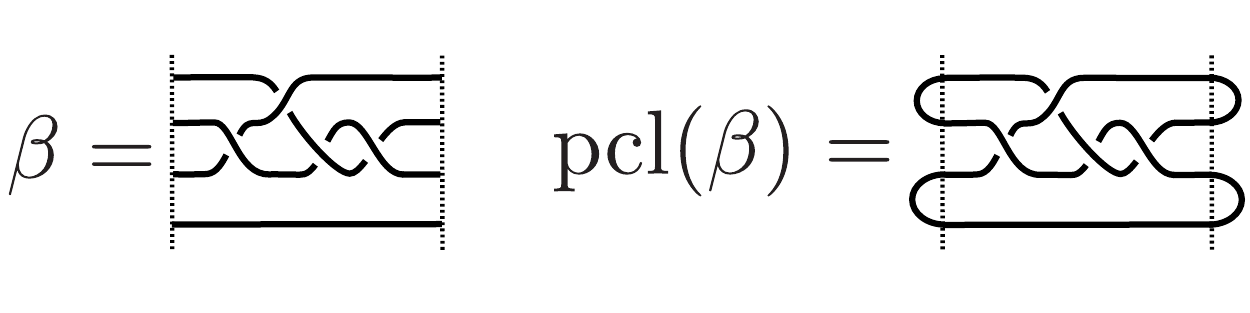}
    \caption{Example of a braid $\beta$ and its plat closure $\mathrm{pcl}(\beta)$. }
    \label{20240707-1}
\end{figure}

For a loop $f: (I, \partial I) \to (\mathcal{W}_n, \omega_0)$, we define the $2n$-braid $\beta_f$ by $\beta_f = \bigcup_{t \in I} \pi(\partial f(t)) \times \{t\} \subset D^2 \times I$, where $\pi: D^2 \times \{0\} \to D^2$ is the projection. 
Let $\partial_*: \pi_1(\mathcal{W}_n, \omega_0) \to B_{2n}$ be the homomorphism given by $\partial _*([f])=[\beta_f]$ for a loop $f$ representing an element of $\pi_1(\mathcal{W}_n, \omega_0)$. 
 The homomorphism $\partial_*$ is injective and the image of $\partial_*$ is called the \textit{Hilden subgroup} of $B_{2n}$, denoted by $K_{2n}$, which is generated by
 braids as follow \cite{Brendle-Hatcher2008}:
\begin{align*}
\sigma_{2i-1}, \quad 
 \sigma_{2j}\sigma_{2j-1}\sigma_{2j+1}\sigma_{2j}, \quad 
 \sigma_{2j}\sigma_{2j-1}\sigma^{-1}_{2j+1}\sigma^{-1}_{2j},
\end{align*}
where $i \in \{1, \dots, n\}$ and $j \in \{ 1, \dots, n-1\}$; see Figure~\ref{20240707-2}.
 We call an element of $K_{2n}$ an {\it adequate braid} and one of the generators an \textit{elementary} adequate braid.

\begin{figure}[h]
    \centering
    \includegraphics[width = 0.6\hsize]{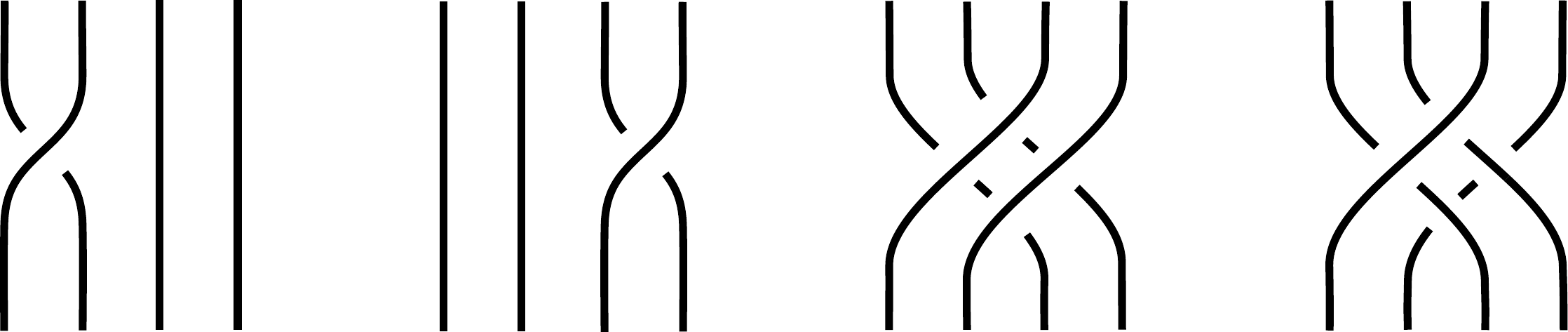}
    \caption{Elementary adequate braids presented by $\sigma_{2i-1}$ (left), $\sigma_{2j}\sigma_{2j-1}\sigma_{2j+1}\sigma_{2j}$ (middle), and $\sigma_{2j}\sigma_{2j-1}\sigma^{-1}_{2j+1}\sigma^{-1}_{2j}$ (right). We omit the $k$th strings for $k \neq 2j-1, 2j, 2j+1, 2j+2$.} 
    \label{20240707-2}
\end{figure}

\subsection{Stabilization of $2n$-braids}
A \textit{stabilizer} of length $n$ is an $n$-tuple  $\lambda$ consisting of non-negative integers.
For a stabilizer $\lambda = (l_1, \dots, l_n)$, we put 
$\mu_i = n + \sum_{j=1}^i l_i$ ($i = 0, 1, \dots, n$), and we put $\nu_\lambda = \mu_n$.
We denote by $\gamma_i$ an adequate $2\nu_\lambda$-braid as follows:
\[
\gamma_i=\sigma_{2i}\sigma_{2i+1}\sigma_{2i-1}\sigma_{2i}  \in K_{2\nu_\lambda}.
\]
Then we define $2\nu_\lambda$-braids $T_{i}^\lambda$ and $T^\lambda$ by:
\begin{align*}
   & T_i^\lambda=~ \prod_{k=i}^{n-1} \gamma_{k} \cdot \prod_{k=n}^{\mu_{i-1}-1} \gamma_{k}^{-1} \in K_{2\nu_\lambda} \quad (i \in \{1, 2, \ldots, n\}),\\
  &  T^\lambda =~ \prod_{i=0}^{n-1} T_{i+1}^\lambda \cdot
    \sigma_{2\mu_{i}}\sigma_{2(\mu_{i}+1)}\cdots \sigma_{2(\mu_{i+1}-1)}\,
   \cdot  (T_{i+1}^\lambda)^{-1} \in K_{2\nu_\lambda}. 
\end{align*}
For a $2n$-braid $\beta$ and a stabilizer $\lambda$ of length $n$, we  take the inclusion
$\iota: B_{2n} \to B_{2 \nu_\lambda}$
sending $\sigma_i$ to $\sigma_i$ $(i=1, \ldots, 2n-1)$. 
 The \textit{$\lambda$-stabilized braid} of $\beta$, denoted by $\beta^\lambda$, is the $2 \nu_\lambda$-braid given by $\beta^\lambda = \iota(\beta) \cdot T^\lambda$.
We also say that $\beta^\lambda$ is obtained from $\beta$ by \textit{$\lambda$-stabilization} or simply \textit{stabilization}.
Note that the plat closures $\mathrm{pcl}(\beta)$ and $\mathrm{pcl}(\beta^\lambda)$ are isotopic in $\mathbb{R}^3$.

\begin{example}
    We take a stabilization of length $n$, $\lambda = (0, \dots, 0, m)$, for a positive integer $m$.
    Then, for a $2n$-braid $\beta$, the $\lambda$-stabilized braid $\beta^\lambda$ is the $2(n+m)$-braid $\iota(\beta) \cdot \sigma_{2n} \sigma_{2n} \dots \sigma_{2(n+m-1)}$, where $\iota: B_{2n} \to B_{2(n+m)}$ is the inclusion.
    We call this stabilization a \textit{canonical stabilization}.
\end{example}

The set of stabilizations of length $n$ has a directed partial ordering $\preceq$ defined as follows:
For stabilizations $\lambda = (l_1, \dots, l_n)$ and $\lambda' = (l'_1, \dots, l'_n)$, $\lambda \preceq \lambda'$ if $l_i \leq l'_i$ for $i = 1, \dots, n$.

\begin{proposition}[\cite{Yasuda21}-Proposition~4.2]\label{Proposition: Birman's criterion for links in plat forms}
    Let $\beta_1$ and $\beta_2$ be a $2n_1$-braid and a $2n_2$-braid, respectively.
    Then, the plat closure $\mathrm{pcl}(\beta_1)$ is isotopic to the plat closure $\mathrm{pcl}(\beta_2)$ by an ambient isotopy of $\mathbb{R}^3$ if and only if there exists a stabilizer $\lambda$ of length $n_1$ satisfying the following condition:
    For any stabilizer $\lambda_1$ satisfying $\lambda \preceq \lambda_1$, there exists a stabilizer $\lambda_2$ of length $n_2$ with $\nu_{\lambda_2} = \nu_{\lambda_1}=:n$ such that $\beta_1^{\lambda_1}$ and $\beta_2^{\lambda_2}$
    belong to the same double coset of $B_{2n}$ modulo the Hilden subgroup $K_{2n}$.
\end{proposition}

\subsection{Banded link presentations for surfaces properly embedded in $\mathbb{R}^4_+$}

Let $L$ and $O$ be links in $\mathbb{R}^3$ such that the union $L \cup O$ is a split union of $L$ and $O$, and let $\mathcal{B}$ be a set of mutually disjoint bands attaching to $L \cup O$.

We denote by $h(L \cup O; \mathcal{B})$ the link obtained from $L \cup O$ by band surgery along bands of $\mathcal{B}$.
A triple $(L, O, \mathcal{B})$ is called \textit{admissible} if both $O$ and $h(L \cup O; \mathcal{B})$ are trivial links.

Let $(L, O, \mathcal{B})$ be an admissible triple and $\mathbf{D}$ (respectively $\mathbf{d}$) the union of mutually disjoint disks in $\mathbb{R}^3$ with $\partial \mathbf{D} = h((L \cup O; \mathcal{B})$ (respectively $\partial \mathbf{d} = O$).
Let $p: \mathbb{R}^4_+ = \mathbb{R}^3\times [0, \infty) \to [0, \infty)$ be the projection.
For a surface $F$ properly embedded in $\mathbb{R}^3 \times [0, \infty)$, we denote by $F_{[t]}$ the slice $p(F \cap (\mathbb{R}^3 \times \{t\}))$.
We define the standard \textit{realizing surface} of $(L, O, \mathcal{B})$, denoted by $F(L,O, \mathcal{B})$, as
a surface $F$ properly embedded in $\mathbb{R}^4_+$ given by:
\begin{align*}
    F_{[t]} ~=~
    \begin{cases}
           L  & (0 \leq t < 1),\\
        L \cup \mathbf{d}   & (t = 1),\\
         L \cup O            & (1 < t < 2),\\
         L \cup O \cup (\cup\mathcal{B})   & (t = 2),\\
        h(L \cup O; \mathcal{B})       & (2 < t < 3),\\
        \mathbf{D}          & (t = 3),\\
        \emptyset           & (t>3),
    \end{cases}
\end{align*}
where $\cup \mathcal{B}$ is the union of bands of $\mathcal{B}$.
Regard $\mathbb{R}^4_+$ as a fiber bundle $\mathbb{R}^3 \times [0, \infty)$ with the fiber $\mathbb{R}^3$.
Let $F'$ be a surface related to $F(L, O, \mathcal{B})$ by a fiber-preserving ambient isotopy of $\mathbb{R}^4_+$. We denote by $(L', O', \mathcal{B}')$ the admissible triple related to $(L, O, \mathcal{B})$. Then, $F'$ is also called a {\it realizing surface} of $(L', O', \mathcal{B}')$, denoted by the same notation $F(L', O', \mathcal{B}')$.
For a realizing surface $F(L, O, \mathcal{B})$, let $t_*$ be the $I_4$ coordinate such that the slice at $t=t_*$ contains the union of the bands $\cup \mathcal{B}$. Then, the set of disks properly embedded in $\mathbb{R}^3 \times (0, t_*]$ (respectively $\mathbb{R}^3 \times [t_*, \infty)$)) whose boundaries are the trivial link $O$ (respectively $h(L \cup O; \mathcal{B}))$ is called a {\it trivial disk system} with respect to $O$ (respectively $h(L \cup O; \mathcal{B})$).
It is known \cite{KSS} that the isotopy class of $F(L, O, \mathcal{B})$ does not depend on the choice of trivial disk systems. It is also known \cite{KSS} that if two admissible triples $(L, O, b)$ and $(L', O', b')$ are isotopic, then their realizing surfaces are isotopic.
In the proof of Theorem \ref{Theorem: Alexander theorem for surfaces in 4-ball}, we also use the following.

\begin{theorem}[\cite{KSS, Meier}]\label{Theorem: KSS for compact surface}
    Every surface properly embedded in $\mathbb{R}^4_+$ is isotopic  to a realizing surface of some admissible triple $(L, O, \mathcal{B})$.
\end{theorem}
For a surface $F$ properly embedded in $\mathbb{R}^4_+$, we call an admissible triple $(L, O, \mathcal{B})$ a {\it banded link presentation} of $F$ if the associated realizing surface $F(L, O, \mathcal{B})$ is isotopic to $F$. 
We also use the following lemma. For a link $L$ in $\mathbb{R}^3$ and a set $\mathcal{B}$ of mutually disjoint bands attaching to $L$, we call the pair $(L, \mathcal{B})$ a {\it banded link} in $\mathbb{R}^3$  \cite{Rudolph}.  

\begin{lemma}[\cite{Yasuda21}-Lemma~4.6]\label{Lemma: Deformation to normal banded braid form}
    By using an isotopy of $\mathbb{R}^3$, any banded link $(L,\mathcal{B})$ in $\mathbb{R}^3$ is deformed to a banded link $(L_0, \mathcal{B}_0)$ satisfying the following conditions:
    \begin{enumerate}[$(1)$]
        \item There exists a disk $D$ in $\mathbb{R}^2$ and a $2m_0$-braid $\beta_0$ in $D\times I$ for some positive integer $m_0$ such that $\beta_0 = L_0\cap (D\times I)$ and $\mathrm{pcl}(\beta_0) = L_0$.
        \item There exist mutually disjoint $m$ subcylinders $U_i$ ($i = 1, \ldots, m$) in $D \times I$ such that each $U_i$ contains a part of $L_0$ as a pair of vertical line segments and a half-twisted band $b_i \in \mathcal{B}_0$ as in Figure~\ref{20241017}, where $m$ is the number of bands belonging to $\mathcal{B}_0$ and a subcylinder of $D\times I$ means a product of a 2-disk $d$ in $D$ and a closed interval $[s,t]$ in $I$.
    \end{enumerate}
\end{lemma}

\begin{figure}[h]
    \centering
    \includegraphics[height=1.5cm]{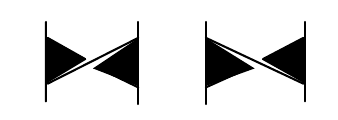}
    \caption{The part $(L_0 \cup b_i) \cap U_i$.}
    \label{20241017}
\end{figure}

In the proof of Theorem~\ref{Theorem: Alexander theorem for surfaces in 4-ball}, we identify $\mathbb{R}^4_+$ with $\mathrm{Int}(D^2\times I_3) \times [0,1)$ to apply these facts.

\subsection{Basic knitted surfaces}
In this subsection, we introduce several braided surfaces and knitted surfaces.
We recall that for a knitted surface $S$ and $t \in I_4$, we denote by $S_{[t]}$ the slice $p(S \cap (D^2 \times I_3 \times \{t\}))$ of $S$.
Furthermore, in this subsection and the proof of Theorem~\ref{Theorem: Alexander theorem for surfaces in 4-ball}, we assume that every knitted surface $S$ of degree $n$ satisfies $S \cap (D^2 \times \partial I_3 \times I_4) = Q_n \times \partial I_3 \times I_4$.
Namely, we assume that the slice $S_{[t]}$ is a knit for each $t \in I_4$.

\subsubsection{Braided surfaces associated with stabilization}

Let $\lambda$ be a stabilizer of a $2n$-braid. 
Let $\mathcal{B}$ be the set of bands attached to the braid $T^\lambda=\prod_{i=0}^{n-1} T_{i+1,\, \mu_i-1}\, \cdot
    \sigma_{2\mu_i}\sigma_{2(\mu_i+1)}\cdots \sigma_{2(\mu_{i+1}-1)} \cdot T_{i+1,\, \mu_i-1}^{-1}$ such that each band corresponds to  one $\sigma_{2j}$ of $\sigma_{2\mu_i}\sigma_{2(\mu_i+1)}\cdots \sigma_{2(\mu_{i+1}-1)}$ and $h(T^\lambda; \mathcal{B})=\prod_{i=0}^{n-1} T_{i+1,\, \mu_i-1} \cdot T_{i+1,\, \mu_i-1}^{-1}$ is equivalent to $e$.
    We take an isotopy $\{\beta_t\}_{t \in [0,1/4]}$ such that $\beta_0=e$ and $\beta_{1/4}=h(T^\lambda; \mathcal{B})$.
We define a braided surface \textit{associated with $\lambda$-stabilization} by the braided surface $S$ of degree $2\nu_\lambda$ given by
\begin{align*}
    S_{[t]} ~=~
    \begin{cases}
           \beta_t  & (0 \leq t <\frac{1}{4}),\\
                      h(T^\lambda; \mathcal{B})   & (\frac{1}{4}\leq t < \frac{1}{2}),\\
        T^\lambda   \cup (\cup \mathcal{B})    & (t = \frac{1}{2}),\\
        T^\lambda       & (\frac{1}{2} <t \leq 1),
    \end{cases}
\end{align*}
see Figure~\ref{20241012-1-2}.
 Further, we define a braided surface \textit{associated with $\lambda$-destabilization} by the surface $S'$ given by $S'_{[t]}=S_{[1-t]}$ for $t \in I$.

\begin{figure}[h]
    \centering
    \includegraphics[width = 0.7\hsize]{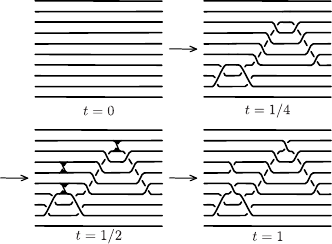}
    \caption{A motion picture of a braided surface associated with $\lambda$-stabilization for $\lambda = (2,1)$. }
    \label{20241012-1-2}
\end{figure}

\subsubsection{Knitted surfaces associated with adequate braids}
Let $\gamma$ be an elementary adequate $2n$-braid, that is, $\gamma$ is either $\sigma_{2i-1}$, $\sigma_{2j}\sigma_{2j-1}\sigma_{2j+1}\sigma_{2j}$, or $\sigma_{2j}\sigma_{2j-1}\sigma^{-1}_{2j+1}\sigma^{-1}_{2j}$ for some $i \in \{1, \dots, n\}$ and $j \in \{1, \dots, n-1\}$.
Then, we define the knitted surface \textit{associated with $\gamma$} is as the knitted surface $S$ given by
\begin{align*}
    S_{[t]} ~=~
    \begin{cases}
           \prod_{k=1}^{n} \tau_{2k-1} \cdot \gamma  & (t=0),\\
            \prod_{k=1}^{n} \tau_{2k-1}  & (t=1),
    \end{cases}
\end{align*}
 and $S_{[t]}$, $t \in I$ are as in  Figures~\ref{20240604-4}, \ref{20240604-2}, or \ref{20240604-3}.  

By using an isotopy of $D^2 \times B^2$ fixing $\partial (D^2 \times B^2)$ pointwise, $S$ is deformed to the surface $S$ described as in Figures~\ref{20240604-4}, \ref{20240609-1}, or \ref{20240609-2} so that $S$ does not have maximal/minimal disks and bands in $S_{[t]}$ for $t \in I_4$. This can be shown by considering 3-cell moves; the detail is left to the reader.

\begin{figure}[h]
    \centering
    \includegraphics{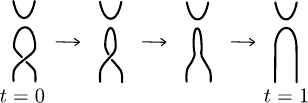}
    \caption{A motion picture of a knitted surface associated with $\gamma = \sigma_{2i-1}$ ($i \in \{1, \dots, n\}$), where we omit the $j$th strings for $j \neq i, i+1$.}
    \label{20240604-4}
\end{figure}

\begin{figure}[h]
    \centering
    \includegraphics[width = \hsize]{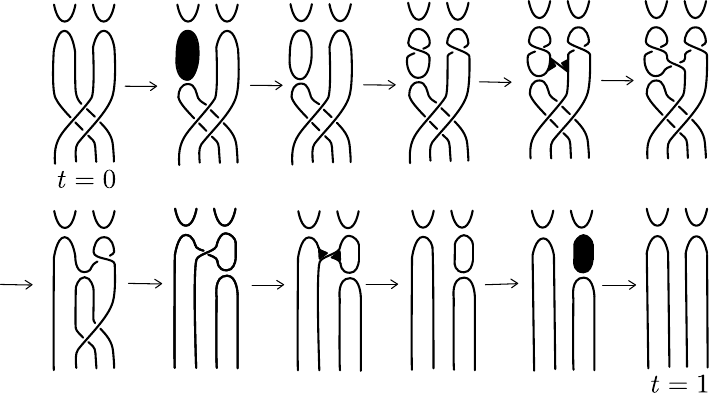}
    \caption{A motion picture of a knitted surface associated with $\gamma = \sigma_{2j}\sigma_{2j-1}\sigma_{2j+1}\sigma_{2j}$ ($j \in \{1, \dots, n-1\}$), where we omit the $k$th strings for $k \notin \{2i-1, 2i, 2i+1, 2i+2\}$. }
    \label{20240604-2}
\end{figure}

\begin{figure}[h]
    \centering
    \includegraphics[width = \hsize]{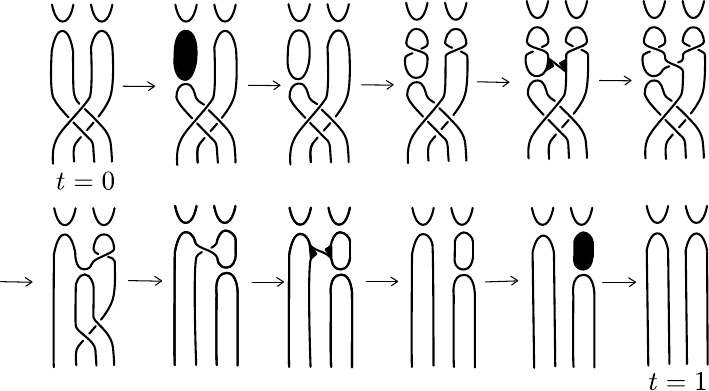}
    \caption{A motion picture of a knitted surface associated with $\gamma = \sigma_{2j}\sigma_{2j-1}\sigma_{2j+1}^{-1}\sigma_{2j}^{-1}$ ($j \in \{1, \dots, n-1\}$), where we omit the $k$th strings for $k \notin \{2i-1, 2i, 2i+1, 2i+2\}$.}
    \label{20240604-3}
\end{figure}

Let $\gamma$ be an adequate braid presented by $\gamma=\gamma_1, \dots, \gamma_m$, where $\gamma_i$ $(i=1,\ldots, m)$ is an elementary adequate braid.
Let $S_i$ be the knitted surface associated with $\gamma_i$ $(i=1,\ldots, m)$, and let $S_i'$ $(i=1,\ldots, m)$ be the knitted surface given by
\[
(S_i')_{[t]}= (S_i)_{[t]}\cdot \prod_{j=i+1}^m \gamma_j   \quad (t\in[0,1]).
\]

 We remark that 
 \begin{eqnarray*}
&&(S_1')_{[0]} = \prod_{k=1}^{n} \tau_{2k-1} \cdot (\gamma_1\cdots \gamma_m), \\
& &(S_i')_{[1]} =(S_{i+1})_{[0]}=\prod_{k=1}^{n} \tau_{2k-1} \cdot \prod_{j=i+1}^m \gamma_j \quad (i=1, \ldots, m-1),\\
&& (S_m')_{[1]}  = \prod_{k=1}^{n} \tau_{2k-1}.
 \end{eqnarray*}
Then we define a knitted surface \textit{associated with $\gamma$} by the knitted surface $S$ given by a  horizontal product $S=S'_1 \cdot S'_2 \cdots S'_m$.

For an elementary adequate braid $\gamma = \sigma_{2j-1}$ ($j \in \{1, \dots, n\}$), the associated knitted surface $S$ does not have maximal/minimal disks or bands.
For an elementary adequate braid $\gamma = \sigma_{2j}\sigma_{2j-1}\sigma_{2j+1}\sigma_{2j}$ (respectively $\sigma_{2j}\sigma_{2j-1}\sigma^{-1}_{2j+1}\sigma^{-1}_{2j}$) ($j \in \{1, \dots, n-1\}$), the associated knitted surface $S$ has one maximal disk, one minimal disk, and two bands;
by using an isotopy of $D^2 \times B^2$ fixing $\partial (D^2 \times B^2)$ pointwise, $S$ is deformed to a surface as in Figure~\ref{20240609-1} (respectively  Figure~\ref{20240609-2}) which does not have  maximal/minimal disks and bands.
Therefore, we have the following proposition.

\begin{proposition}\label{Proposition: Elementary knitted surface associated with adequate braid}
    For any adequate braid $\gamma$, a knitted surface $S$ associated with $\gamma$ is ambiently isotopic rel $\partial$ to a surface $S'$ without maximal/minimal disks and bands.
\end{proposition}

\begin{figure}[h]
    \centering
    \includegraphics*{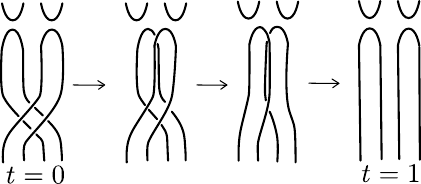}
    \caption{A motion picture from $\tau (\sigma_{2i}\sigma_{2i-1}\sigma_{2i+1}\sigma_{2i})$ to $\tau$ ($i \in \{1, \dots, n-1\}$), where we omit the $j$th strings for $j \notin \{2i-1, 2i, 2i+1, 2i+2\}$.}
    \label{20240609-1}
\end{figure}

\begin{figure}[h]
    \centering
    \includegraphics*{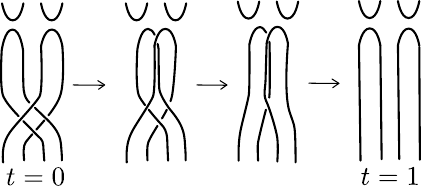}
    \caption{A motion picture from $\tau (\sigma_{2i}\sigma_{2i-1}\sigma_{2i+1}^{-1}\sigma_{2i}^{-1})$ to $\tau$ ($i \in \{1, \dots, n-1\}$), where we omit the $j$th strings for $j \notin \{2i-1, 2i, 2i+1, 2i+2\}$. }
    \label{20240609-2}
\end{figure}

\subsection{Proof of Theorem~\ref{Theorem: Alexander theorem for surfaces in 4-ball}}
Let $F$ be a surface properly embedded in $D^2 \times B^2 = D^2 \times I_3 \times I_4$.
By Theorem~\ref{Theorem: KSS for compact surface}, $F$ is  isotopic to a realizing surface of an admissible triple $(L, O, \mathcal{B})$.
We construct a knitted surface $S_*$ in $D^2\times B^2$ from $(L, O, \mathcal{B})$ and then we show that $S_*$ and $F$ are isotopic.

We construct $S_*$ in such a way as follows.
Let $n_*$ be the number of components of the trivial link $O$, $\beta_0$ a braid such that $\mathrm{pcl}(\beta_0) \cong L$.
Put $2n_0 =\deg (\beta_0)$ and $n_1 = n_0 + n_*$.
Let $\beta_1$ be the braid obtained from $\beta_0$ by adding trivial $2n_*$ strands: $\beta_1=\iota(\beta_0)$, where $\iota$ is the natural inclusion $\iota: B_{2n_0} \to B_{2n_1}$.
Applying Lemma~\ref*{Lemma: Deformation to normal banded braid form}, we obtain a $2n_2$-braid $\beta_2$ and a set $\mathcal{B}_0$ of bands attaching to $\beta_2$ such that $(\mathrm{pcl}(\beta_2), \mathcal{B}_0) \cong (L \cup O, \mathcal{B})$ and each band of $\mathcal{B}_0$ is in the form as described in Lemma \ref{Lemma: Deformation to normal banded braid form} (2).
Let $n_3$ be the number of components of the trivial link $h(\mathrm{pcl}(\beta_2); \mathcal{B}_0)$ and let $\beta_3$ be the trivial $2n_3$-braid. 
Notice that $\mathrm{pcl}(\beta_1)$ and $\mathrm{pcl}(\beta_3)$ are isotopic to $L \cup O$ and $h(L \cup O; \mathcal{B})$, respectively.

We apply Proposition~\ref{Proposition: Birman's criterion for links in plat forms} to two pairs $(\beta_1, \beta_2)$, $(h((\beta_2; \mathcal{B}_0), \beta_3)$ so that we obtain an integer $n \geq \{n_1, n_2, n_3\}$, adequate $2n$-braids $\gamma'$, $\gamma$, $\delta$, $\delta'$ and stabilizers $\lambda_i$ ($i = 1, 2, 3$) such that $\nu_{\lambda_i }= n$ ($i = 1, 2, 3$) and
\[
    \beta_1^{\lambda_1} ~=~ \gamma\,\beta_2^{\lambda_2}\,\gamma', \quad
   h (\beta_2;\mathcal{B}_0)^{\lambda_2} = \delta\,\beta_3^{\lambda_3} \,\delta' \quad \mathrm{in} \quad B_{2n},
\]
where $h(\beta_2; \mathcal{B}_0)$ is the braid  obtained from $\beta_2$ by band surgery along bands of $\mathcal{B}_0$.
Further, we take $\lambda_0 = (0, \dots, 0, n - n_1)$, a canonical stabilizer. 

For an integer $1 \leq i < j \leq n$, we put $\tau_{(i,j)} = \prod_{k = i}^{j} \tau_{2k-1}$ and $\tau = \tau_{(1, n)}$. 
We construct knitted surfaces $S_1, S_2, \dots, S_{10}$ of degree $2n$, and then take their horizontal product. We define $S_1, \ldots, S_{10}$ as follows (see Figures~\ref{20240609-3-f} and \ref{20240609-3-l}).

\noindent (1) 
Let $S$ be the braided surface associated with $\lambda_0$-destabilization.
Then, we define the knitted surface $S_1$ by
\[
(S_1)_{[t]}= (S)_{[t]}\cdot \iota(\beta_0)   \quad (t\in[0,1]),
\]
where $\iota: B_{2n_0} \to B_{2n}$ is the natural inclusion.  
The slice $(S_1)_{[1/2]}$ is the union of the braid $\iota(\beta_0)$ and attached bands.
We denote the set of bands in $(S_1)_{[1/2]}$ by $\Delta_0$.

\noindent (2)
Let $S$ be the knitted surface presented by
\[
e \to \tau,
\]
and let $S'$ be the knitted surface presented by
\[
e \to \tau_{(1, n_1)}.
\]
Then, we define $S_2$ by
\[
(S_2)_{[t]}= S_{[t]} \cdot \iota(\beta_0) \cdot S'_{[t]}   \quad (t\in[0,1]).
\]
We remark that  $(S_2)_{[1]} = \tau \cdot \iota(\beta_0) \cdot \tau_{(1, n_1)}$.
We denote by $\Delta_1$ the union of bands corresponding to saddle points of $S_2$.

\noindent (3)
Let $S$ be the knitted surface presented by
\[
\tau \to \tau \cdot \tau_{(n_1+1, n)}.
\]
Then, we define $S_3$ by
\[
(S_3)_{[t]}= S_{[t]} \cdot \iota(\beta_0)  \cdot \tau_{(1, n_1)} \quad (t\in[0,1]).
\]
We remark that $(S_3)_{[1]} = \tau \cdot  \tau_{(n_1+1, n)} \cdot \iota(\beta_0) \cdot \tau_{(1, n_1)}$.

\noindent (4)
The knits $\tau_{(n_1+1, n)}$ and $\iota(\beta_0)\cdot \tau_{(1, n_1)}$ are commutative, that is, $\tau_{(n_1+1, n)} \cdot \iota(\beta_0) \cdot \tau_{(1, n_1)} \cong \iota(\beta_0) \cdot \tau_{(1, n_1)} \cdot \tau_{(n_1+1, n)}$.
Let $S_4$ be the knitted surface presented by
\[
\tau \cdot  \tau_{(n_1+1, n)} \cdot \iota(\beta_0) \cdot \tau_{(1, n_1)} \to
 \tau \cdot  \iota(\beta_0) \cdot \tau.
\]

\noindent (5)
Let $S$ be the knitted surface associated with $\lambda_1$-stabilization of $\beta_1$. Remark that $S_{[0]}=\iota(\beta_1)=\iota (\beta_0) \in B_{2n}$, where the first $\iota$ (respectively the second $\iota$) are the inclusion from  $B_{2n_0}$ (respectively $B_{2n_1}$) to $B_{2n}$.
Then we define $S_5$ by
\[
(S_5)_{[t]}=  \tau  \cdot S_{[t]} \cdot \tau \quad (t\in[0,1]).
\]

\begin{figure}[h]
    \centering
    \includegraphics*[width = \hsize]{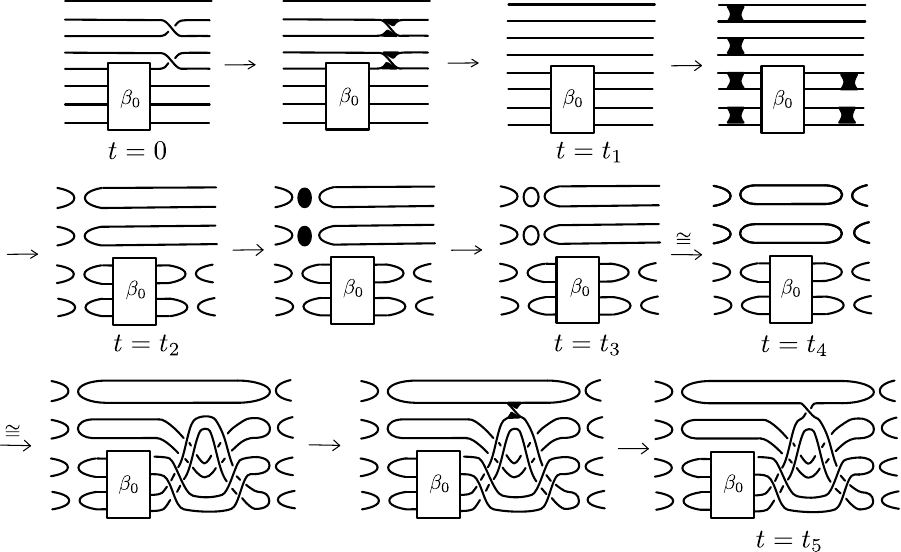}
    \caption{The motion picture of $S$ for $t \in [0, t_5]$, where $(n_0, n_1) = (2,3)$, $\lambda_0 = (0,2)$, and $\lambda_1 = (1,0,0)$. }
    \label{20240609-3-f}
\end{figure}

\noindent (6)
We define
knitted surfaces $S_6^-$ and $S_6^+$ as follows.
Since $\beta_1^{\lambda_1} = \gamma \beta_2^{\lambda_2} \gamma'$, let $S$ be a braided surface without branch points presented by
 \[
 S_{[0]} = \beta_1^{\lambda_1} \to  \gamma \beta_2^{\lambda_2} \gamma' =S_{[1]}.
 \]
Then we define $S_6^-$ by
\[
(S_6^-)_{[t]}=  \tau  \cdot S_{[t]} \cdot \tau \quad (t\in[0,1]).
\]

Next, let $S_-$ be the knitted surface associated with $\gamma$, and let $S_{+}^*$ be the mirror image of the knitted surface associated with $\gamma'$ with respect to $D^2 \times \{1/2\} \times I_4$.
Then we define $S_6^+$ by
\[
(S_6^+)_{[t]}=  (S_-)_{[t]}  \cdot \beta_2^{\lambda_2} \cdot (S_{+}^*)_{[t]} \quad (t\in[0,1]).
\]

Then we define $S_6$ as the horizontal product
\[
S_6=S_6^- \cdot S_6^+.
\]

\noindent (7)
We define $S_7$ by
\begin{align*}
    (S_7)_{[t]} ~=~
    \begin{cases}
        \tau \beta_2^{\lambda_2} \tau & (1 \leq t < \frac{1}{2}).\\
        \tau \beta_2^{\lambda_2} \tau  \cup (\cup \mathcal{B}_0) & (t = \frac{1}{2}),\\
        \tau \cdot h(\beta_2^{\lambda_2}; \mathcal{B}_0) \cdot \tau     & (\frac{1}{2} < t \leq 1).
    \end{cases}
\end{align*}

\noindent (8)
We construct $S_8$ by a method similar to (6).
Since $h(\beta_2^{\lambda_2}; \mathcal{B}_0) = \delta \beta_3^{\lambda_3} \delta'$,
let $S$ be the knitted surface
presented by
 \[
 S_{[0]} =h (\beta_2^{\lambda2}; \mathcal{B}_0)\overset{\cong}{\to} \delta \beta_3^{\lambda_3} \delta' =S_{[1]}.
 \]
Then we define $S_8^-$ by
\[
(S_8^-)_{[t]}=  \tau \cdot S_{[t]}  \cdot \tau \quad (t\in[0,1]).
\]
Next, let $S_-$ be the knitted surface associated with $\delta$, and let $S_{+}^*$ be the mirror image of the knitted surface associated with $\delta'$ 
with respect to $D^2 \times \{1/2\} \times I_4$.
Then we define $S_8^+$ by
\[
(S_8^+)_{[t]}=  (S_-)_{[t]}  \cdot \beta_3^{\lambda_3} \cdot (S_{+}^*)_{[t]} \quad (t\in[0,1]).
\]
We denote by $\Delta_3$ the union of bands corresponding to the saddle points of $S_8$.
Then we define $S_8$ as the horizontal product
\[
S_8 = S_8^- \cdot S_8^+.
\]

\noindent (9)
Let $S$ be the braided surface associated with $\lambda_3$-destabilization of $e \in B_{n_3}$, the trivial braid of $n_3$ strands.
We define $S_9$ by
\[
(S_9)_{[t]} = \tau \cdot  S_{[t]}  \cdot \tau \quad (t\in[0,1]).
\]

\noindent (10)
We define $S_{10}$ as the knitted surface presented by
\[
\tau \tau \to  \tau.
\]

\begin{figure}[h]
    \centering
    \includegraphics*[width = \hsize]{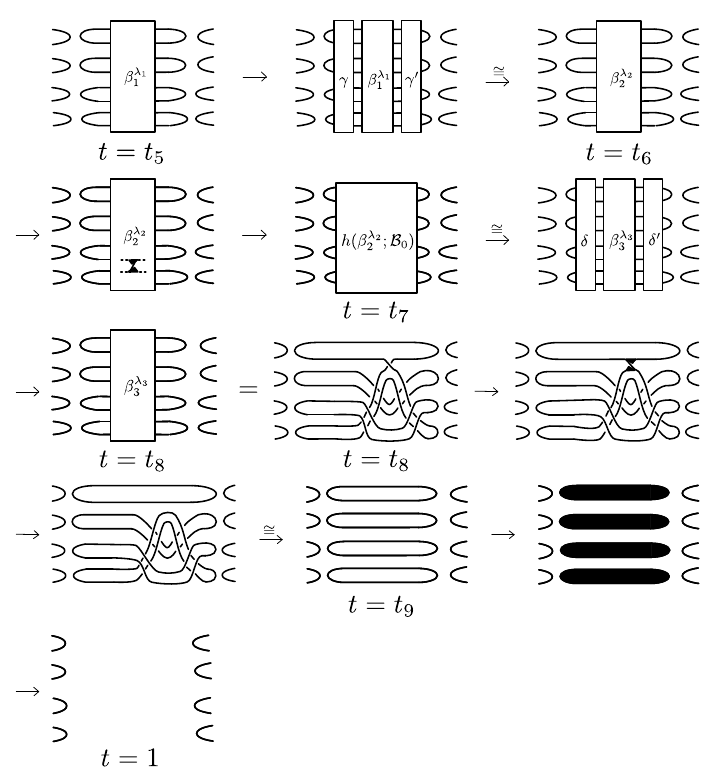}
    \caption{The motion picture of $S$ for $t \in [t_5, 1]$, where $n_3=2$ and $\lambda_3 = (1,1)$. }
    \label{20240609-3-l}
\end{figure}

Let $S_*$ be the horizontal product of $S_1, \dots, S_{10}$:
\[
S_*=S_1 \cdot S_2 \cdots S_{10}.
\]
Let $0 = t_0 < t_1 < \dots < t_{10} = 1$ be the partition of $I_4$ such that $S_* \cap (D^2\times I_3 \times [t_{i-1}, t_i]) = S_i$ for each $i \in \{ 1, \dots, 10\}$.

Next we show that $S_*$ is isotopic to a realizing surface of the admissible triple $(L, O, \mathcal{B})$. We put $S=S_*$.
We arrange the $I_4$ coordinates so that
 maximal/minimal disks and saddle bands of $S_*$ are contained in slices $S_{[t]}$ $(t \in I_4)$ as follow:
\begin{itemize}
    \item Saddle bands $\Delta_0$ at $t = t_1/2$.
    \vspace{0.1cm}
    \item Saddle bands $\Delta_1$ at  $t = t_1$.
    \vspace{0.1cm}
    \item Minimal disks $\mathbf{d}$ at $t = (t_2+t_3)/2$.
    \vspace{0.1cm}
    \item  Saddle bands $\Delta_2$ at $t=(t_4+t_5)/2$.
    \vspace{0.1cm}
    \item Saddle bands and maximal/minimal disks at several $t \in [t_5, t_6]$.
     \vspace{0.1cm}
    \item Saddle bands $\mathcal{B}_0$ at $t=(t_6+t_7)/2$.
    \vspace{0.1cm}
    \item Saddle bands and maximal/minimal disks at several $t \in [t_7, t_8]$.
    \vspace{0.1cm}
    \item Saddle bands $\Delta_3$ at $t=(t_8+t_9)/2$.
    \vspace{0.1cm}
    \item Maximal disks $\mathbf{D}$ at $t=(t_9+t_{10})/2$.
\end{itemize}

\noindent \textbf{(Step 1)}
Let $\tau_{(i,j)}^+$ (respectively $\tau_{(i,j)}^-$) be the union of hooks of $\tau_{(i,j)}$ containing minimal points (respectively  maximal points).
We put
\[
\tau^+ = \tau_{(1,n)}^+, \ \tau^- = \tau_{(1,n)}^-.
\]
We see that $\tau \beta \tau=\tau^- \cup \mathrm{pcl}(\beta) \cup \tau^+$ for any knit $\beta$, where $\mathrm{pcl}(\beta)$ denotes the plat closure. We present the knit $\tau^- \cup \mathrm{pcl}(\beta) \cup \tau^+$ by $\tau^- \cdot  \mathrm{pcl} (\beta) \cdot \tau^+$.

For each $t \in [t_4, t_9]$, there exists a unique $2n$-knit $\beta'_t$ such that the slice $S_{[t]}$ is in one of the forms as follow:
\[
S_{[t]}= \begin{cases}
\tau^- \cdot \mathrm{pcl}(\beta'_t)\cdot  \tau^+ \\
\tau^- \cdot \mathrm{pcl}(\beta'_t)\cdot  \tau^+ \cup (\cup \mathcal{B})  \\
 \tau^- \cdot \mathrm{pcl}(\beta'_t)\cdot  \tau^+ \cup \mathbf{D}, \quad
 \end{cases}
 \]
 where $\mathcal{B}$ is a set of attaching bands and $\mathbf{D}$ is a union of maximal/minimal disks; see Figure \ref{20241013-4}. 
By arranging the $I_3$ coordinates, we assume that $S$ is in the form so that there exists a small $s_*>0$ 
satisfying
\begin{align*}
  &  S \cap (D^2 \times [0,s_*]\times [t_4, 1]) = \tau^- \times [t_4, 1], \\
  &  S \cap (D^2 \times [1-s_*, 1]\times [t_4, 1]) = \tau^+ \times [t_4, 1].
\end{align*}

\begin{figure}[h]
    \centering
    \includegraphics{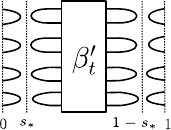}
    \caption{The slice  $S_{[t]}=\tau^- \cdot \mathrm{pcl}(\beta_t') \cdot \tau^+$. }
    \label{20241013-4}
\end{figure}

\noindent \textbf{(Step 2)}
We consider slices $S_{[t]}$ for $t \in [(t_5+t_6)/2, t_6]$ which describes the knitted surface $S_6^-$ constructed in (6).
By Proposition~\ref{Proposition: Elementary knitted surface associated with adequate braid},
we deform $S$ by an isotopy relative to the complement of $\mathrm{Int}(D^2 \times [s_*, 1-s_*] \times [(t_5+t_6)/2, t_6])$, so that $S_{[t]}$ has no maximal/minimal disks and bands for $t \in [t_5, t_6]$.
Similarly, using Proposition~\ref{Proposition: Elementary knitted surface associated with adequate braid}, we deform $S$ by an isotopy relative to the complement of $\mathrm{Int}(D^2 \times [s_*, 1-s_*] \times [(t_7+t_8)/2, t_8])$, so that $S_{[t]}$ has no maximal/minimal disks and bands for $t \in [t_7, t_8]$.

\noindent \textbf{(Step 3)}
We consider $S_{[t]}$ for $t \in [0, t_4]$.
We assume that $\Delta_0$ is contained in $S_{[t_1/2]}\cap(D^2 \times (s_*, 1- s_*) \times \{t_1/2\})$.
Then we deform $S$ by an isotopy of $D^2 \times I_3 \times [t_1, t_4]$ so that $S$ satisfies
\begin{align*}
S \cap (D^2 \times [0, s_*] \times ((t_1+t_2)/2, t_4]) = \tau_- \times ((t_1+t_2)/2, t_4], \\
S \cap (D^2 \times [1-s_*, 1] \times ((t_1+t_2)/2, t_4]) = \tau_+ \times ((t_1+t_2)/2, t_4].
\end{align*}
Let $C_0$ be the union of 3-cells given by $\Delta_0 \times [(t_0 + t_1)/2, t_1]$ and $C_1$ the union of 3-cells given by $\Delta_1 \times [t_1, (t_1 + t_2)/2]$.
Then we deform $S$ by 3-cell moves along $C_0$ and $C_1$ so that the slice of $S$ for $t \in [0, t_4]$ is as follows (see Figure~\ref{20240610-1}):
\begin{align*}
    S_{[t]} ~=~
    \begin{cases}
        \beta_0^{\lambda_0} & (0 \leq t < t_1),\\
        \tau^-\cdot \mathrm{pcl}(\beta_0) \cdot \tau^+ \cup \Delta_0 \cup \Delta_1 & (t = t_1),\\
        \tau^-\cdot \mathrm{pcl}(\beta_0) \cdot \tau^+
         &(t_1 < t < \frac{t_2 + t_3}{2}),\\
        \tau^-\cdot \mathrm{pcl}(\beta_0) \cdot \tau^+ \cup \mathbf{d} & (t = \frac{t_2 + t_3}{2}),\\
        \tau^-\cdot \mathrm{pcl}(\beta_0) \cdot \tau^+ & (\frac{t_2 + t_3}{2} < t \leq t_4). \\
    \end{cases}
\end{align*}

\begin{figure}[h]
    \centering
    \includegraphics*[width = \hsize]{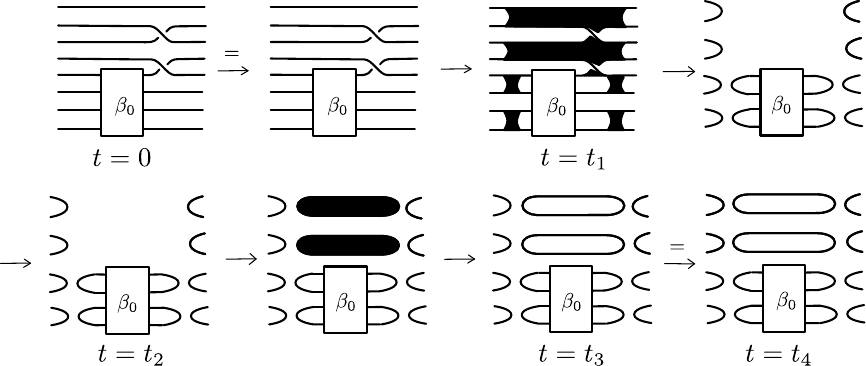}
    \caption{The motion picture of $S$ for $t \in [0, t_4]$. }
    \label{20240610-1}
\end{figure}

\noindent \textbf{(Step 4)}
We consider $S_{[t]}$ for $t \in [t_2, t_5]$.
Let $L'$ be a trivial link in $D^2 \times I_3 \times \{(t_4 + t_5)/2\}$ such that $L'$ is the subset of $S_{[(t_4+t_5)/2]}$ coming from the trivial link $\partial \,\mathbf{d}$ in $S_{[(t_2+t_3)/2]}$ by the isotopy determined from the slices $\{S_{[t]}\}$.
Let $\mathbf{d}'$ be a union of 2-disks such that $\partial \, \mathbf{d}' = L'$ and $\mathrm{Int}( \mathbf{d}')$ is disjoint from $S_{[(t_4 + t_5)/2]}$.
Put $\mathbf{d}'_2 = \mathbf{d}' \cup \Delta_2$.
Then, $\mathbf{d}'_2$ is a union of minimal disks and 2-cells attaching to $\mathrm{pcl}(\beta'_{(t_4 + t_5)/2})$. 
We define the subset $\mathbf{d}_2$ of $\mathbf{d}$ 
 as the union of minimal disks of $\mathbf{d}$ corresponding to the union of minimal disks of $\mathbf{d}'$ through the slices $\{S_{[t]}\}$.
The slices $\{S_{[t]}\}$ for $t \in [(t_2 + t_3)/2, (t_4 + t_5)/2]$ and the union of 2-cells of $\mathbf{d}'$ give an isotopy carrying $\mathrm{pcl}(\beta_1)$ to $\mathrm{pcl}(\beta_1^{\lambda_1})$. 
Hence, we deform $S \cap (D^2 \times I_3 \times [t_2, t_5])$ by an  isotopy rel $\partial$ so that
\begin{align*}
    S_{[t]} ~=~
    \begin{cases}
        \beta_0^{\lambda_0} & (0 \leq t < t_1),\\
        \tau^-\cdot \mathrm{pcl}(\beta_0) \cdot \tau^+ \cup \Delta_0 \cup \Delta_1 & (t = t_1),\\
        \tau^-\cdot \mathrm{pcl}(\beta_0) \cdot \tau^+ &(t_1 < t < \frac{t_2 + t_3}{2}),\\
        \tau^-\cdot \mathrm{pcl}(\beta_0) \cdot \tau^+ \cup \mathbf{d}_2 &(t = \frac{t_2 + t_3}{2}),\\
        \tau^- \cdot L_t \cdot \tau^+ & (\frac{t_2 + t_3}{2} < t \leq t_5),\\
    \end{cases}
\end{align*}
where $\{L_t\}_{t \in [(t_2 + t_3)/2, t_5]}$ is the isotopy with  $L_{[(t_2+t_3)/2, t_5]}=\mathrm{pcl}(\beta_1)$ induced from the slices and the union of 2-cells (see Figure~\ref{20241014-3}).

\begin{figure}[h]
    \centering
    \includegraphics*[width = \hsize]{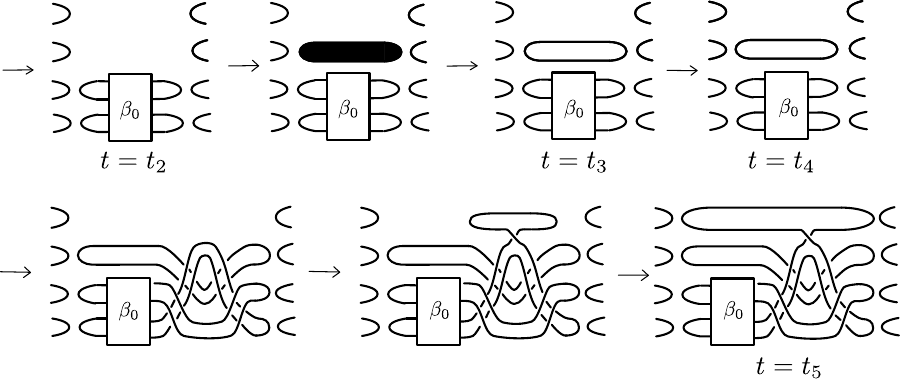}
    \caption{The motion picture of $S$ for $t \in [t_2, t_5]$. }
    \label{20241014-3}
\end{figure}

\noindent \textbf{(Step 5)}
We consider $S_{[t]}$ for $t \in [t_8, t_{10}]$.
Put $\beta'_t=\emptyset$, the empty set, for $t \in (t_9, t_{10}]$. Then, together with the braids $\beta'_t$ $(t \in [t_8, t_9])$ given in Step 1, $S_{[t]} = \tau \beta'_t \tau$ for $t \in [t_8, t_{10}]$. 
We remark that $\beta'_{(t_8 + t_9)/2} = T^{\lambda_3} \cup \Delta_3$. 
Let $\mathbf{D}'$ be a union of mutually disjoint disks in $D^2\times I_3 \times \{t\}$ for $t = (t_8 + t_9)/2$ such that $\mathbf{D}' \cap (\tau \beta'_t \tau) = \partial \mathbf{D}' = \mathrm{pcl}(h(T^\lambda_3;\Delta_3))$
and $\mathrm{Int}( \mathbf{D}')$ do not intersect with $\Delta_3$. 
We put $\mathbf{D}_2 =\Delta_3 \cup \mathbf{D}'$.
We define a surface $S'$ properly embedded in $D^2 \times I_3 \times [t_8, t_{10}]$ by
\begin{align*}
    S'_{[t]} ~=~
    \begin{cases}
        \tau \beta'_t \tau & (t_8 \leq t < \frac{t_8 + t_9}{2}),\\
        (\tau \beta'_t \tau) \cup \mathbf{D}_2 & ( t = \frac{t_8 + t_9}{2}),\\
        \tau & (\frac{t_8 + t_9}{2} < t \leq t_{10}).
    \end{cases}
\end{align*}
Let $\mathbf{D}_*$ and $\mathbf{D}'_*$ be the components of $S \cap D^2 \times I_3 \times [t_8, t_{10}]$ associated with $\mathbf{D}$ and $\mathbf{D}'$, respectively. The sets $\mathbf{D}_*$ and $\mathbf{D}'_*$ are trivial disk systems of $\mathrm{pcl}(h(T^{\lambda_3}; \Delta_3))$ in $D^2 \times I_3 \times [t_8, t_{10}]$; hence they are isotopic in $D^2 \times I_3 \times [t_8, t_{10}]$ rel $D^2 \times I_3 \times \{t_8\}$ (see \cite{KSS}).
We deform $S$ by an isotopy so that $S_{[t]} = S'_{[t]}$ for $t \in [t_8, t_{10}]$.
 
\noindent \textbf{(Step 6)}
Maximal/minimal disks and saddle bands of the original $S=S_*$ are now in the form of 2-disks contained in slices $S_{[t]}$ $(t \in I_4)$ as follow:
\begin{itemize}

\item Saddle bands and 2-disks $\Delta_0 \cup \Delta_1$ at  $t=t_1$.
\vspace{0.1cm}
\item Minimal disks $\mathbf{d}_2$ at $t=(t_2+t_3)/2$.
\vspace{0.1cm}
\item
Saddle bands $\mathcal{B}_0$ at $t = (t_6 + t_7)/2$.
\vspace{0.1cm}
\item Maximal disks $\mathbf{D}_2$ at $t=(t_8+t_{9})/2$. 
\end{itemize}

Let $C$ be a union of 3-cells in $D^2 \times B^2$ defined as follows.
A \textit{shadow disk} of a simple arc $l$ properly embedded in $D^2 \times I$ is a 2-disk $d$ embedded in $D^2 \times I$ such that $\partial d$ is a union of $l$ and a simple arc in $\partial (D^2 \times I)$.
Let $\mathbf{d}^+$ (respectively $\mathbf{d}^-$) be the union of connected components of a disjoint union of shadow disks of $\tau^+$ (respectively $\tau^-$) in $D^2 \times [1-s_*, 1]$ (respectively $D^2 \times [0, s_*]$) such that $\Delta_0 \cup \Delta_1$ is disjoint from the interior of $\mathbf{d}^+$ (respectively $\mathbf{d}^-$).
Then we define $C$ by
\[
C= (\mathbf{d}^+ \cup  \mathbf{d}^-)\times I_4 \,\cup\, (\Delta_0 \cup \Delta_1) \times [0, t_1], 
\]
see Figure \ref{20241013-5}. 
\begin{figure}[h]
    \centering
    \includegraphics[width = \hsize]{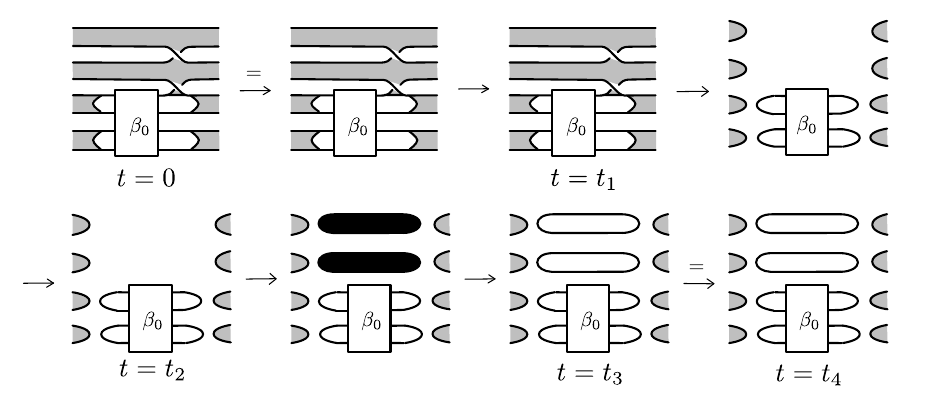}
    \caption{The motion picture of the 3-cell $C$ (gray regions). }
    \label{20241013-5}
\end{figure}
We see that the intersection of $S$ and each component of $C$ is a 2-disk.
We deform $S$ by 3-cell moves along $C$ so that $\partial S = \mathrm{pcl}(\beta_0) \times \{0\} \subset (D^2 \times I_3) \times I_4$.
Notice that these moves delete 2-disks $\Delta_0 \cup \Delta_1$ in the slice $S_{[t_1]}$.

Therefore, the resulting surface, also denoted by $S$, is described by
\begin{align*}
    S_{[t]} =
    \begin{cases}
        \mathrm{pcl}(\beta_0) & (0 \leq t < \frac{t_2 + t_3}{2}),\\
        \mathrm{pcl}(\beta_0) \cup \mathbf{d} & (t = \frac{t_2 + t_3}{2}),\\
        \mathrm{pcl}(\beta'_t) & (\frac{t_2 + t_3}{2} < t < \frac{t_6 + t_7}{2}),\\
       \mathrm{pcl}(\beta'_t) \cup (\cup \mathcal{B}_0) & (t = \frac{t_6 + t_7}{2}),\\
        \mathrm{pcl}(\beta'_t) & (\frac{t_6 + t_7}{2} < t < \frac{t_8 + t_9}{2}),\\
        \mathrm{pcl}(\beta'_t) \cup \mathbf{D} & (t = \frac{t_8 + t_9}{2}),\\
        \emptyset & (\frac{t_8 + t_9}{2} < t \leq 1), 
    \end{cases}
\end{align*}
see Figures \ref{20241014-1} and \ref{20241014-2}. 

\begin{figure}[h]
    \centering
    \includegraphics[width = \hsize]{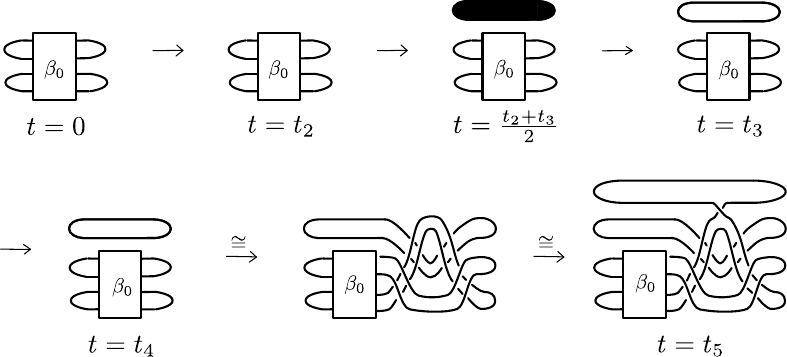}
    \caption{The motion picture of $S$ for $t \in [0,t_5]$. }
    \label{20241014-1}
\end{figure}

\begin{figure}[h]
    \centering
    \includegraphics[width = \hsize]{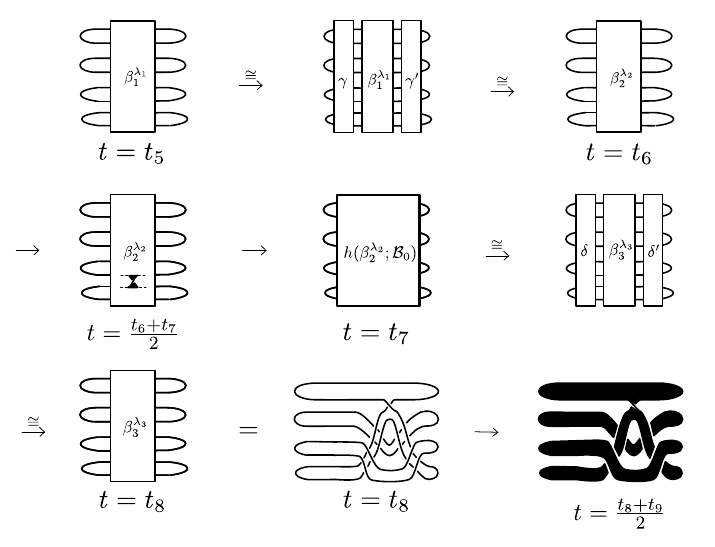}
    \caption{The motion picture of $S$ for $t \in [t_5, (t_8+t_9)/2]$. }
    \label{20241014-2}
\end{figure}

We see that $\mathrm{pcl}(\beta'_{(t_6 + t_7)/2}) = \mathrm{pcl}(\beta_2^{\lambda_2})$.
We denote by $L_*$ and $O_*$ the links isotopic to $\mathrm{pcl}(\beta_0)$ and the trivial link $\partial (\cup \mathbf{d})$, respectively, related by the isotopy $\{S_{[t]}\}$, such that $L_* \cup O_*=\mathrm{pcl}(\beta_2^{\lambda_2})$.
The knitted surface $S$ is a realizing surface of the admissible triple $(L_*, O_*, \mathcal{B}_0)$.
Since $(\mathrm{pcl}(\beta_2^{\lambda_2}), \mathcal{B}_0)
\cong (L \cup O, \mathcal{B})$, it follows that
$(L_*, O_*, \mathcal{B}_0) \cong (L, O, \mathcal{B})$, and hence $S=F(L_*, O_*, \mathcal{B}_0) \cong F(L, O, \mathcal{B}) \cong F$.
Thus, any surface $F$ is isotopic to a knitted surface.
\qed

\section{BMW chart moves}\label{sec7-1}
In this section, we give the definition of BMW chart moves, that are local moves of BMW charts of degree $n$ such that their associated BMW charts are equivalent. This includes C-moves of charts of 2-dimensional braids \cite{Kamada02, Kamada96, Kamada92}. See also \cite{Kamada-Matumoto} for charts and C-moves for (non-simple) braided surfaces.

For a BMW chart $\Gamma$ of degree $n$, we denote by $\Gamma_\tau$ the subgraph of $\Gamma$ consisting of $\tau$-edges and the connecting vertices, ignoring the labels of edges, which will be called the {\it $\tau$-chart }of $\Gamma$.
Regarding diagonal edges around each crossing as one edge so that each crossing is not a vertex but the intersection point of a pair of edges, we regard $\Gamma_\tau$ as consisting of intervals, trivalent graphs and circles, and their subgraphs obtained from taking the  intersection with $B^2$, each of which will be called a {\it component} of $\Gamma_\tau$.

\begin{definition}\label{def6-2}
Let $\Gamma$ and $\Gamma'$ be BMW charts of degree $n$ in a 2-disk $B^2$.
Then, a {\it BMW chart move}, or simply a {\it chart move} or a {\it C-move}, is a local move in a 2-disk $N \subset B^2$ such that $\Gamma$ (respectively $\Gamma'$) is changed to $\Gamma'$ (respectively $\Gamma$), satisfying the following.

\begin{enumerate}
\item
The boundary $\partial N$ does not contain vertices of $\Gamma$ and $\Gamma'$, and $\Gamma \cap \partial N$ (respectively $\Gamma' \cap \partial N$) consists of transverse intersection points of edges of $\Gamma$ (respectively $\Gamma'$) and $\partial N$.

\item
The charts are identical in the complement of $N$:
$\Gamma \cap (B^2\backslash N)=\Gamma' \cap (B^2\backslash N)$.

\item
The charts in $N$, $\Gamma \cap N$ and $\Gamma' \cap N$, satisfy one of the following.
\begin{enumerate}
\item[(T)]
C-moves for $\tau$-charts.

The charts $\Gamma \cap N$ and $\Gamma' \cap N$ consist of $\tau$-edges and vertices: $\Gamma \cap N=\Gamma_\tau \cap N$ and $\Gamma' \cap N=\Gamma'_\tau \cap N$, and they are as in Figure \ref{T1-T4}.
We call the chart moves (T1)--(T6), respectively.
In particular, we call (T2) a {\it YU-move}, and (T4) an {\it IH-move}.

\item[(CI)]
C-moves for charts with no black $\sigma$-vertices.

The charts $\Gamma \cap N$ and $\Gamma' \cap N$ do not contain black $\sigma$-vertices, and each component of the $\tau$-charts $\Gamma_\tau \cap N=\Gamma'_\tau \cap N$ has an intersection with $\partial N$.

\item[(CII)]
C-moves for charts with a black $\sigma$-vertex (1).

The chart $\Gamma \cap N$ contains a black $\sigma$-vertex and a crossing, and the chart $\Gamma' \cap N$ is the result of moving the black $\sigma$-vertex to delete the crossing, as depicted in the upper figure in Figure \ref{Chartmove23}, where the orientations are allowed for all possible cases:
the orientations of $\sigma$-edges in the figure are for example.

\item[(CIII)]
C-moves for charts with a black $\sigma$-vertex (2).

The chart $\Gamma \cap N$ contains a black $\sigma$-vertex and
a white $\sigma$-vertex, and the chart $\Gamma' \cap N$ is the result of moving the black $\sigma$-vertex to delete the white vertex, as depicted in the lower figure in Figure \ref{Chartmove23}, where the orientations are allowed for all possible cases:
the orientations of $\sigma$-edges in the figure are for example.
 \end{enumerate}
\end{enumerate}

We call the chart moves {\it T, CI, CII, CIII-moves}, respectively.
\end{definition}

\begin{figure}[ht]
\includegraphics*[height=8cm]{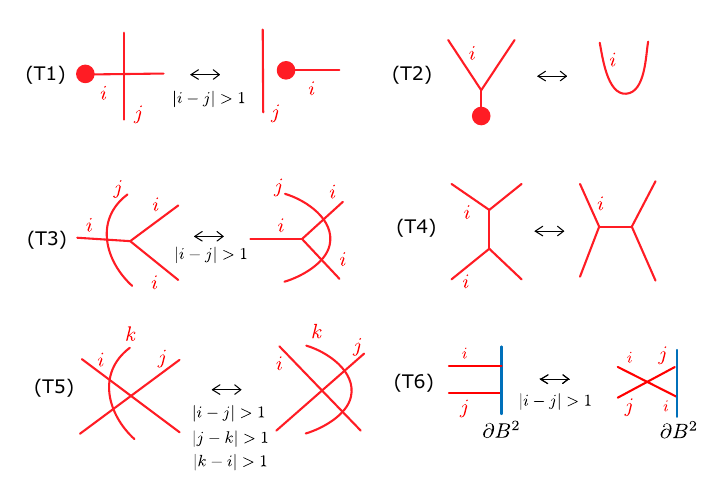}
\caption{T-moves (T1)--(T6). }
\label{T1-T4}
\end{figure}

\begin{figure}[ht]
\includegraphics*[height=5cm]{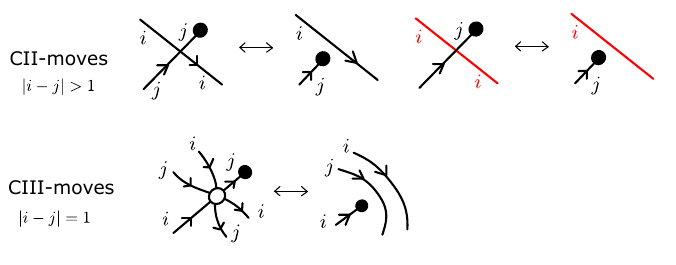}
\caption{CII-moves and CIII-moves, where the orientations are for example. }
\label{Chartmove23}
\end{figure}

\begin{definition}
Let $\Gamma$ and $\Gamma'$ be BMW charts of degree $n$ in a 2-disk $B^2$.
We say that $\Gamma$ and $\Gamma'$ are {\it BMW chart-move equivalent} or simply {\it C-move equivalent} if they are related by a finite sequence of BMW chart moves and ambient isotopies of $B^2$.
\end{definition}

We give a set of explicit CI-moves as follows; see Claim \ref{claim9-2}.  

\begin{definition}\label{c1moves}
We define chart moves (A1)--(A9) (generators of $\sigma$-CI-moves), (B1)--(B3), (C1)--(C3), (D1)--(D10), (E1)--(E3) and (F1)--(F3), as follow,  where the orientations are allowed for all possible cases:
the orientations of $\sigma$-edges in the figures are for example.

\begin{enumerate}[(A)]
\item
CI-moves for $\sigma$-edges.
\\
The local moves as in Figure \ref{Chartmove1}. There are 9 types.
They are generators of CI-moves for the original charts of braided surfaces.
\item
CI-moves for charts with a black $\tau$-vertex.
\\
The local moves as in Figure \ref{B1-B3}. There are 3 types.
\item
CI-moves for charts with a boundary point.
\\
The local moves as in Figure \ref{C1-C32}. There are 3 types.
\item
CI-moves for charts with a $\tau$-edge.
\\
The local moves as in Figure \ref{D1-D7}. There are 10 types.

\item
 CI-moves for charts with a trivalent $\tau$-vertex.
\\
The local moves as in Figure \ref{E1-E3}. There are 3 types.
\item
CI-moves for a $\tau$-crossing.
\\
The local moves as in Figure \ref{F1-F32}. There are 3 types.
\end{enumerate}
\end{definition}

\begin{figure}[ht]
\includegraphics*[height=11cm]{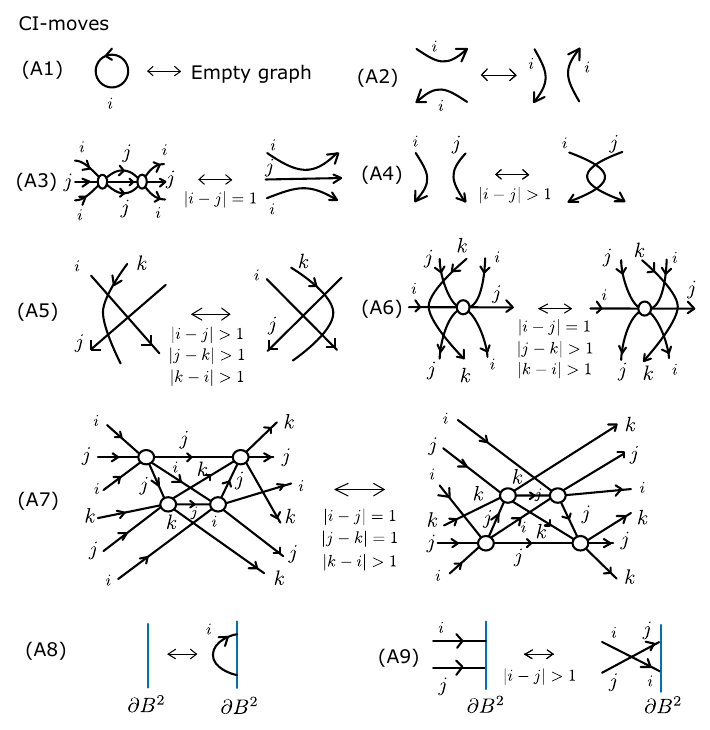}
\caption{CI-moves (A1)--(A9), where the orientations are for example. }
\label{Chartmove1}
\end{figure}

\begin{figure}[ht]
\includegraphics*[height=6cm]{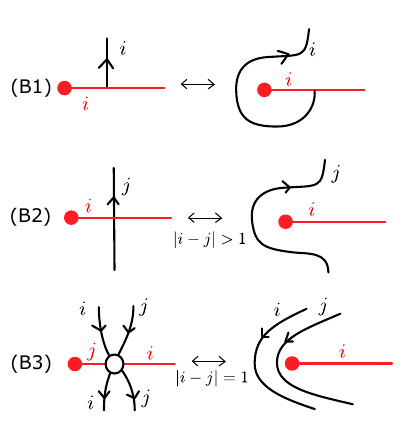}
\caption{CI-moves (B1)--(B3), where the orientations are for example. }
\label{B1-B3}
\end{figure}

\begin{figure}[ht]
\includegraphics*[height=6cm]{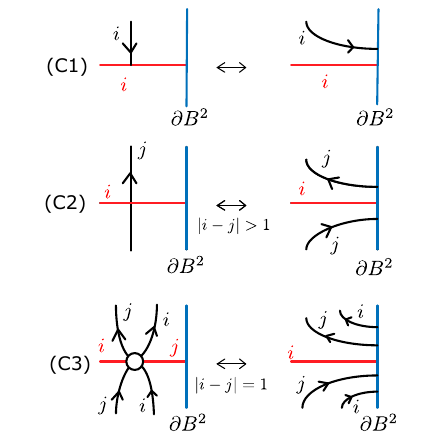}
\caption{CI-moves (C1)--(C3), where the orientations are for example.}
\label{C1-C32}
\end{figure}

\begin{figure}[ht]
\includegraphics*[height=13cm]{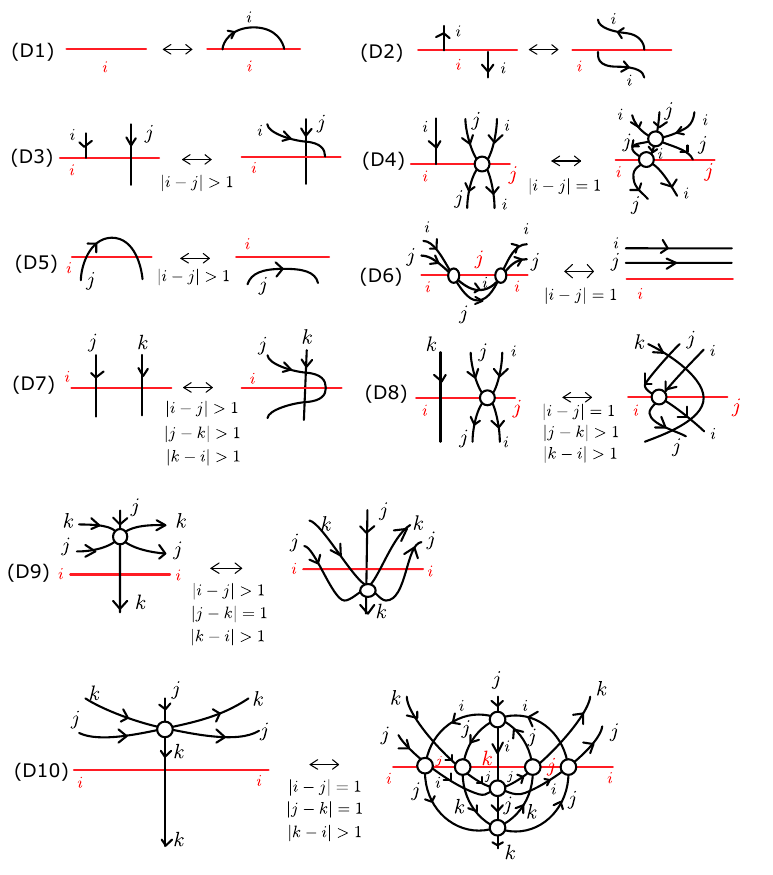}
\caption{CI-moves (D1)--(D10), where the orientations are for example.}
\label{D1-D7}
\end{figure}

\begin{figure}[ht]
\includegraphics*[height=5cm]{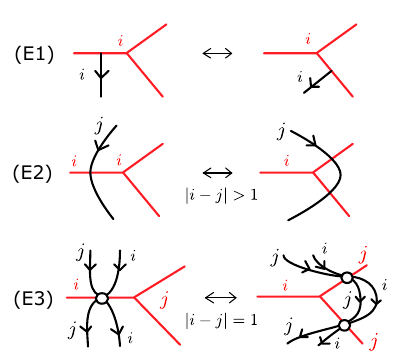}
\caption{CI-moves (E1)--(E3), where the orientations are for example.}
\label{E1-E3}
\end{figure}

\begin{figure}[ht]
\includegraphics*[height=8cm]{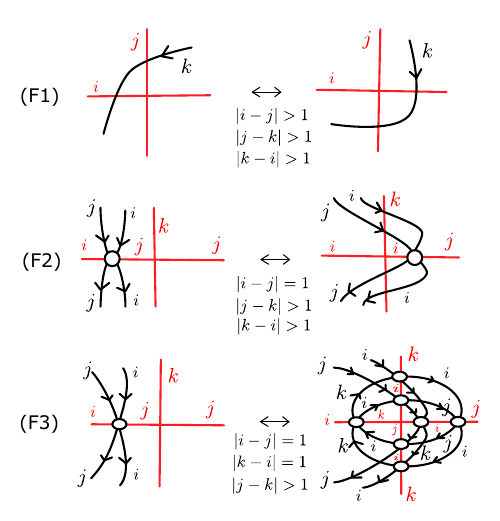}
\caption{CI-moves (F1)--(F3), where the orientations are for example.}
\label{F1-F32}
\end{figure}

\begin{remark}

 The (D10) and (F3)-moves are analogues of the (A7)-move: the chart obtained from (D10)/(F3) by changing $\tau$-edges into $\sigma$-edges are essentially the same with the (A7)-move; we leave the detailed
 explanation as an exercise to the reader.
 The (D5)-move is obtained as a combination of the (E2)-move and the YU-move (namely the (T2)-move).
 \end{remark}

\begin{theorem}\label{thm7-3}
Two knitted surfaces of degree $n$ are equivalent 
if their presenting charts are related 
by a finite number of C-moves consisting of T, CII, CIII-moves and 
CI-moves of type (A1)--(F3). 
\end{theorem}
 
 \begin{proof}
 Let $\Gamma$ and $\Gamma'$ be charts of degree $n$ related by one of C-moves in question, and let $S$ and $S'$ be the knitted surfaces of degree $n$ presented by $\Gamma$ and $\Gamma'$, respectively. 
 If $\Gamma$ and $\Gamma'$ are related by one of the original C-moves, which is the case when $\Gamma$ and $\Gamma'$ do not contain $\tau$-edges, then $S$ and $S'$ are equivalent \cite{Kamada02, Kamada92}; hence, there exists an isotopy of braided surfaces connecting $S$ and $S'$. Since there is a surjection from the set of $n$-braids to the set of $n$-knits which is given by changing a crossing to a hook pair, the isotopy induces an isotopy of knitted surfaces presented by C-moves analogous to the original C-moves. Here, a C-move analogous to an original C-move $\mathcal{R}$ is the local move obtained from $\mathcal{R}$ by changing several $\sigma$-edges into $\tau$-edges. 
 It remains to show for the cases which involve vertices of degree 3: (T2), (T3), (T4), (B1), (C1), (D1), (D2), (D3), (D4), (E1)--(E3). 
 For these cases, by describing $S$ and $S'$ by motion pictures, we can construct a 3-cell such that $S'$ is obtained from $S$ by a 3-cell move which induces an isotopy of knitted surfaces. Thus $S$ and $S'$ are equivalent, and we have the required result. 
\end{proof}

\section{Knitted surfaces of degree 2 and BMW charts presenting trivial surface-knots}\label{sec8}

In this section, we consider knitted surfaces of degree 2.
We give a chart of degree 2 such that the plat closure of its knitted surface is a given trivial surface-knot (Theorem \ref{prop6-5}). In this section, a chart of degree 2 is called a {\it $2$-chart}.

For a chart, 
we call an edge a {\it free edge} if its endpoints are black vertices. 
We call an edge a {\it half edge} if exactly one endpoint is a black vertex: the other endpoint is either a boundary point  or a trivalent vertex.  
When a half $\sigma$-edge has an orientation toward (respectively from) the black vertex, we call it a {\it positive} (respectively {\it negative}) half $\sigma$-edge.

Let $S$ be a knitted surface of degree 2 in $D^2 \times B^2$ such that the boundary $\partial S \subset D^2 \times \partial B^2$ is the closure of a 2-braid $\beta=\sigma_1^n$  for some integer $n$.
Let $A$ be a band with $n$ half twists, embedded in $D^2 \times \partial B^2$, such that $\partial A=\partial S$.
We define the {\it plat closure} of $S$ by the surface-link $S \cup A \subset \mathbb{R}^4$, determined from the natural inclusion $D^2 \times B^2 \subset \mathbb{R}^4$. Further, for a knitted surface $S$ such that $\partial S$ is the plat closure of a 2-braid, we define the {\it plat closure} of $S$ as the surface-link obtained from $S$ by pasting an embedded disk $A \subset D^2 \times \partial B^2$ to $\partial S$.

A {\it standard 2-sphere} (respectively a {\it standard torus}) in $\mathbb{R}^4$ is the boundary of a 3-ball (respectively an unknotted solid torus) in $\mathbb{R}^3 \times \{0\}$, and a standard positive projective plane (respectively a {\it standard negative projective plane}) is the knotted torus as illustrated in (1) (respectively (2)) in Figure \ref{20240330-2}.
A surface-knot in $\mathbb{R}^4$ is {\it trivial} if it is a connected sum of a finite number of standard surfaces in $\mathbb{R}^4$. 
We denote by $T$ a standard torus, and by $P_+$ (respectively $P_-$) a standard positive (respectively negative) projective plane.

\begin{theorem}\label{prop6-5}
Let $S$ be a trivial surface-knot, which is the connected sum of $g$ copies of $T$ and $m$ copies of $P_+$ and $n$ copies of $P_-$.
Then $S$ is the plat closure of the knitted surface presented by a $2$-chart $\Gamma$ in $B^2$ such that $\Gamma$ consists of $\min\{m,n\}$ copies of free $\sigma$-edges, $g$ copies of free $\tau$-edges,
and $|m-n|$ copies of positive (respectively negative) half $\sigma$-edges connected with $\partial B^2$ if $m \geq n$ (respectively $m<n$).
\end{theorem}

\begin{figure}[ht]
\includegraphics*[height=3cm]{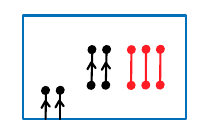}
\caption{Example of a $2$-chart whose plat closure is a  trivial surface-knot, where we omit the labels, which are all one. 
}
\label{20240331-1}
\end{figure}

\begin{example}
Let $\Gamma$ be a 2-chart as given in Figure \ref{20240331-1}. Then, the plat closure of the presented knitted surface is a trivial surface-knot given as a connected sum of 3 copies of $T$, 4 copies of $P_+$ and 2 copies of $P_-$. 
\end{example}

\begin{proof}[Proof of Theorem \ref{prop6-5}]
First we observe that a standard 2-sphere is the plat closure of the 2-chart which is the empty graph.
An addition of a  free $\tau$-edge is an addition of a 1-handle between the first and the second sheets, and since we consider the plat closure, we see that the 1-handle is trivial; thus, taking a connected sum with one copy of $T$ is an addition of a free $\tau$-edge to the 2-chart.

Now we obtain 2-charts of a projective plane and a connected sum of projective planes, as follows.
Essentially our argument here is the same with that given in the proof of \cite[Theorem 5.5]{Yasuda21}.
Seeing Figure \ref{20240330-2}(1) and (2), we see that the middle part of (1) and (2) are presented by the 2-charts (3) and (4) in an annulus, respectively. By a (C1)-move, we move the half $\sigma$-edge so that the 2-charts are C-move equivalent to the 2-charts as in Figure \ref{20240330-3}(1) and (2), respectively. The projective planes are obtained by taking the plat closures, that is, by pasting disks to the boundaries, of the surfaces presented by the 2-charts. Taking a connected sum of projective planes corresponds to taking the sum of the 2-charts as indicated in Figure \ref{20240330-3}(3). Since the knitted surface presented by a chart on an annulus is presented by a chart in a disk as in the right figure of Figure \ref{20240330-5}, we see that projective planes and their connected sum is presented by the 2-charts as in Figure \ref{20240330-4}. Then, by (C1)-moves, we move the half $\sigma$-edges to be connected with $\partial B^2$ and then, since we consider the plat closure, we can regard $\partial B^2$ to be a $\tau$-edge and we remove the half $\tau$-edge by a YU-move.

Thus, the $2$-chart consists of $g$ copies of free $\tau$-edges
and $m$ (respectively $n$) copies of positive (respectively negative) half $\sigma$-edges. Then, applying (A2) and (D1)-moves to obtain free $\sigma$-edges, we have a 2-chart in the required form.
\end{proof}

\begin{figure}[ht]
\includegraphics*[height=6cm]{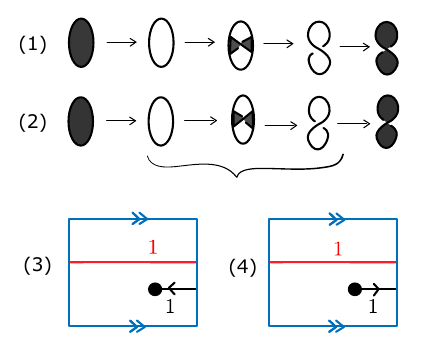}
\caption{(1) A standard positive projective plane $P_+$ and (2) a standard negative projective plane $P_-$.
The middle part of (1) and (2) are presented by the 2-charts (3) and (4) in an annulus, respectively.}
\label{20240330-2}
\end{figure}

\begin{figure}[ht]
\includegraphics*[height=3cm]{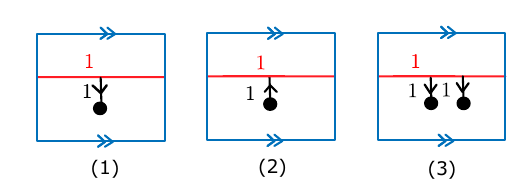}
\caption{Charts of degree 2 in an annulus whose closures present (1) a positive projective plane $P_+$, (2) a negative projective plane $P_-$, and (3) the connected sum $P_+ \# P_+$.}
\label{20240330-3}
\end{figure}

\begin{figure}[ht]
\includegraphics*[height=3cm]{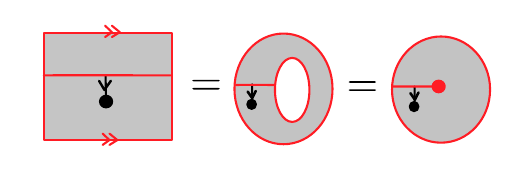}
\caption{Chart descriptions of the same knitted surface, where we focus on the surfaces presented by the charts in the grey regions.}
\label{20240330-5}
\end{figure}

\begin{figure}[ht]
\includegraphics*[height=3cm]{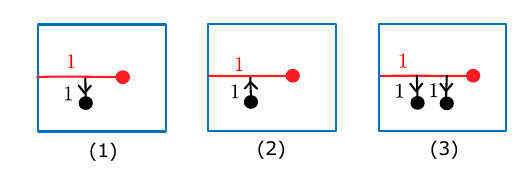}
\caption{Charts of degree 2 in $B^2$ whose plat closures present (1) a positive projective plane  $P_+$, (2) a negative projective plane $P_-$, and (3) the connected sum $P_+ \# P_+$.}
\label{20240330-4}
\end{figure}

We claim the following, which will be shown in a future paper.

\begin{claim}\label{thm6-7}
Let $S$ be the plat closure of the knitted surface presented by a chart of degree 2. Then, $S$ is a trivial surface-link.
\end{claim}

\section{Further results}\label{sec9}

A knitted surface $S$ is called a {\it 2-dimensional $n$-knit} if
$\partial S=Q_n \times \partial B^2$.
Let $S$ be 2-dimensional $n$-knit in $D^2 \times B^2$ and $S_0 = Q_n \times B^2$ be the trivial 2-dimensional $n$-braid.
Then, the {\it closure} of $S$ is a surface-link in $S^4$ defined as the union of $S$ and $S_0$ in $S^4= D^2\times B^2 \cup_\partial D^2 \times B^2$.
It is known \cite{Kamada94-2} that
every orientable surface-link is ambiently isotopic to the closure of some 2-dimensional braid.

In our future papers, we will show the followings. 
\begin{claim}\label{Theorem: Alexander theorem for surface-links}
    Every surface-link in $S^4$ is isotopic to the closure of some 2-dimensional knit.
\end{claim}
This claim will be shown by focusing on the fact that the plat closure $\mathrm{pcl}(\beta)$ of a braid $\beta$ is equivalent to the closure of the knit $\tau \beta$ (see Figure \ref{20241013-2}), and seeing that the analogous deformation holds also for the \lq\lq plat closure'' of braided surfaces \cite{Yasuda21}. \\

\begin{figure}[h]
    \centering
    \includegraphics[height=1.5cm]{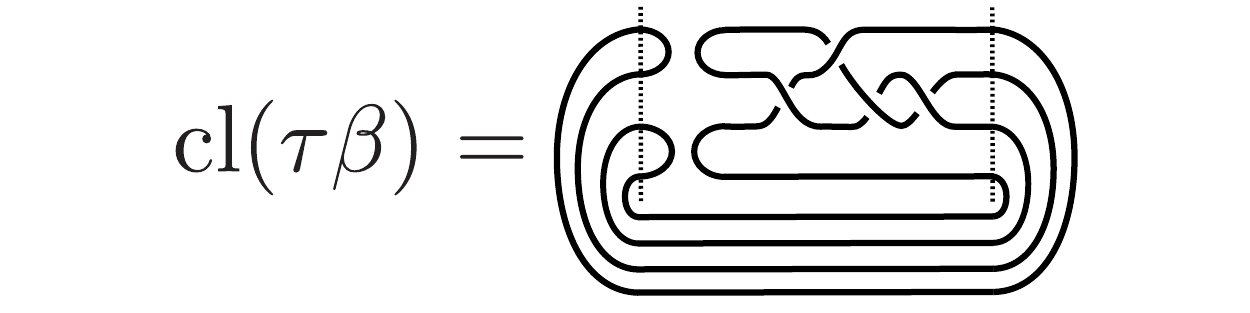}
    \caption{The closure $\mathrm{cl}(\tau\beta)$ of the knit $\tau\beta$, where $\beta$ is the braid of Figure \ref{20240707-1}. }
    \label{20241013-2}
\end{figure}

\begin{claim}\label{claim9-2}
CI-moves are generated by the moves {\rm (A1)--(F3)} given in Definition \ref{c1moves}.
\end{claim}

This claim will be shown by an argument as follows. 
When we have only $\sigma$-edges and vertices, then it is an original CI-move, which is generated by (A1)--(A9).
For the other cases which contain $\tau$-edges, we will  construct a sequence of moves consisting of (A1)--(F3), which relates the BMW charts. 
Further, we expect that 
if two knitted surfaces of degree $n$ are equivalent, then  their presenting charts are C-move equivalent.

\section*{Acknowledgements}
The first author was partially supported by JST FOREST Program, Grant Number JPMJFR202U. The second author was partially supported by JSPS KAKENHI Grant Number 22J20494.

\end{document}